\documentclass[11pt]{amsart}

\usepackage{amsmath} 
\usepackage{amssymb}
\usepackage{hyperref}
\usepackage{amsmath,amsthm,amsxtra,mathtools}
\usepackage{eucal,mathrsfs}
\usepackage{enumitem}
\usepackage{fullpage}
\usepackage{xcolor}
\usepackage{upgreek}
\usepackage{soul}
\usepackage{bbm,stix}

\usepackage{tikz}
\usetikzlibrary{cd}

\newcommand{\N}{\mathbb{N}}
\newcommand{\R}{\mathbb{R}}

\newcommand{\X}{\mathbb{X}}
\newcommand{\Y}{\mathbb{Y}}

\newcommand{\x}{\boldsymbol{x}}

\newcommand{\calB}{\mathcal{B}}

\newcommand{\calD}{\mathcal{D}}
\newcommand{\calE}{\mathscr{E}}
\newcommand{\calF}{\mathcal{F}}

\newcommand{\calH}{\mathcal{H}}

\newcommand{\calL}{\mathcal{L}}
\newcommand{\calM}{\mathcal{M}}

\newcommand{\calP}{\mathcal{P}}

\newcommand{\calR}{\mathcal{R}}

\newcommand{\sfb}{\mathsf{b}}
\newcommand{\sfd}{\mathsf{d}}
\newcommand{\sfh}{\mathsf{h}}
\newcommand{\sfr}{\mathsf{r}}
\newcommand{\frakd}{\mathfrak{d}}

\newcommand{\sfD}{\mathsf{D}}
\newcommand{\sfE}{\mathsf{E}}

\newcommand{\sfP}{\mathsf{P}}

\newcommand{\sfR}{\mathsf{R}}

\newcommand{\sfJ}{\mathsf{J}}
\newcommand{\sfT}{\mathsf{T}}
\newcommand{\sfN}{\mathsf{N}}

\newcommand{\scrE}{\mathscr{E}}
\newcommand{\scrR}{\mathscr{R}}
\newcommand{\scrV}{\mathscr{V}}
\newcommand{\scrS}{\mathscr{S}}
\newcommand{\scrC}{\mathscr{C}}
\newcommand{\scrD}{\mathscr{D}}
\newcommand{\scrI}{\mathscr{I}}

\newcommand{\n}{{\bf n}}
\newcommand{\dd}[1]{\mathop{}\!\mathrm{d} #1}
\newcommand{\Ent}{\mathsf{Ent}}
\newcommand{\dnabla}{\overline{\nabla}}
\newcommand{\ddiv}{\overline{\text{div}}}
\newcommand{\Cyl}{\mathrm{Cyl}}

\newtheorem{mainthm}{Theorem}

\newtheorem{theorem}{Theorem}[section]

\newtheorem{lemma}[theorem]{Lemma}
\newtheorem{proposition}[theorem]{Proposition}

\theoremstyle{definition}
\newtheorem{definition}[theorem]{Definition}
\newtheorem{remark}[theorem]{Remark}
\newtheorem{example}[theorem]{Example}
\newtheorem{assu}[theorem]{Assumption}

\title[Large population limit]{Large population limit of interacting population dynamics via generalized gradient structures}

\author{Jasper Hoeksema}
\address{Department of Mathematics and Computer Science, Eindhoven University of Technology, 5600 MB Eindhoven, the Netherlands}
\email{j.hoeksema@tue.nl}

\author{Anastasiia Hraivoronska}
\address{Universite Claude Bernard Lyon 1, CNRS, Ecole Centrale de Lyon, INSA Lyon, Université Jean Monnet, ICJ UMR5208, 69622 Villeurbanne, France}
\email{hraivoronska@math.univ-lyon1.fr}
\urladdr{https://hraivoronska.flywheelsites.com/}

\author{Oliver Tse}
\address{Department of Mathematics and Computer Science, Eindhoven University of Technology, 5600 MB Eindhoven, the Netherlands}
\email{o.t.c.tse@tue.nl}
\urladdr{https://oliver-tse.site}

\begin{document}

\maketitle

\begin{abstract}
    This chapter focuses on the derivation of a doubly nonlocal Fisher-KPP model, which is a macroscopic nonlocal evolution equation describing population dynamics in the large population limit. The derivation starts from a microscopic individual-based model described as a stochastic process on the space of atomic measures with jump rates that satisfy detailed balance w.r.t.\ to a reference measure. We make use of the so-called `cosh' generalized gradient structure for the law of the process to pass to the large population limit using evolutionary Gamma-convergence. In addition to characterizing the large population limit as the solution of the nonlocal Fisher-KPP model, our variational approach further provides a generalized gradient flow structure for the limit equation as well as an entropic propagation of chaos result.
\end{abstract}

\section{Introduction}

Population dynamics, a fundamental area of ecology, is concerned with the study of complex phenomena triggered by simple, and often, random changes in a given population such as birth, death, immigration, emigration, and mutation, that are driven by environmental conditions, resources, etc, and can occur on separate levels of the biological hierarchy---the individual organism, populations of organisms, and communities of populations \cite{begon2021ecology}. In recent years, fueled by the increase in computational power, individual-based models (IBMs) have become a popular tool for ecologists to understand complex phenomena and how they emerge from the behavior of individual organisms. Yet, while they are able to reproduce complex patterns and are computationally tractable, IBMs are more difficult to analyze and require a different set of mathematical tools than their analytical counterpart in classical ecology \cite{grimm2005ecology}, hence creating the need to derive simplified models that reliably approximate the effects of space and stochasticity in the appropriate scaling regimes \cite{cornell2019unified}.

In this chapter, we discuss and prove the large population limit of IBMs for population dynamics to their macroscopic counterparts. In contrast to other works on similar asymptotic limits of stochastic population models \cite{BaMe2015,champagnat2006unifying,finkelshtein2010vlasov,miroslaw2005}, our approach makes use of the variational structures associated with the IBMs, which is modeled as jump processes on the space of finite positive measures on a compact Polish space $X$. To keep the main message of this exposition clear, we choose only to consider the asymptotic limit of birth-death-and-hopping dynamics for a single population type on a compact subset $X$ of $\R^d$, and mention possible generalizations for several instances. In the large population limit, these measure-valued jump processes are shown to converge to the so-called doubly nonlocal Fisher-Kolmogorov-Petrovsky-Piskunov (F-KPP) equation \cite{FKT2019} 
\begin{equation}\label{eq:mf-intro}\tag{\textsf{F-KPP}}
	\partial_t u_t(x) = \sfr_\sfb(u_t,x) - \sfr_\sfd(u_t,x)\,u_t(x) + \int_X \bigr(u_t(y)-u_t(x)\bigr)\,\sfr_\sfh(u_t,x,y)\,\pi(\dd y),
\end{equation}
where $\pi$ is a given \emph{activity measure} on $\Omega$, the curve $t\mapsto u_t$ describes the evolution of the population density relative to $\pi$, and $\sfr_j$, $j\in\{\sfb,\sfd,\sfh\}$ are density-dependent functions stemming from the birth, death, and hopping rates, respectively, that satisfy an appropriate \emph{detailed-balance} condition. We remark that \eqref{eq:mf-intro} is traditionally called doubly nonlocal in the literature because $\sfr_j$, $j\in\{\sfb,\sfd,\sfh\}$ can depend nonlocally on the density $u_t$ and the equation includes nonlocal hopping dynamic. Since this section is mainly concerned with the derivation of \eqref{eq:mf-intro} from IBMs, we will place strong assumptions on $\sfr_j$, $j\in\{\sfb,\sfd,\sfh\}$, guaranteeing the global well-posedness of the mean-field equation \eqref{eq:mf-intro} by standard ODE results in Banach spaces.

Under the detailed-balance condition, the mean-field equation \eqref{eq:mf-intro} may be seen as a 
\emph{generalized gradient flow} on the space of measures in the sense of \cite{PRST2022}---here, the term `generalized' refers to having a gradient structure with non-quadratic dissipation potential. Various examples of generalized gradient flows have appeared earlier in rate-independent systems \cite{mielke2011}, Markov chains, and mass-action kinetics \cite{grmela2010multiscale,mielke-etal2017,mielke-etal2014}.
Our variational approach derives the generalized gradient structure for \eqref{eq:mf-intro} from the gradient structure of its underlying microscopic IBM by means of \emph{evolutionary Gamma convergence} \cite{mielke2016evolutionary,mielke-etal2012,sandier2004gamma,serfaty2011gamma}, thus justifying the choice of the gradient structure for the mean-field equation \eqref{eq:mf-intro}. An alternative convergence technique is based on EDP convergence \cite{mielke2016evolutionary,liero-etal2017,dondl-etal2019,mielke2021exploring}---with some additional complexity, EDP convergence would provide the \emph{unique} mean-field gradient structure. However, we choose to work with evolutionary Gamma-convergence to keep the presentation focused on the overall perspective of our approach. Earlier works with a similar strategy can be found in \cite{erbaretal2016,maasmielke2020}, and a setting with pure birth and death with simplified $u$-dependence was treated in \cite{hoeksema2023} by two of the authors.

\subsubsection*{Beyond gradient dynamics} Many systems in nature exhibit (generalized) gradient dynamics, though most do not. \emph{Active} particle systems, which describe a multitude of biological processes, form a large family of intrinsically non-gradient systems. In these systems, each agent can extract energy from their environment and convert it into directed motion or mechanical forces \cite{RevModPhys.88.045006,conmatphys-070909-104101}. While this work focuses on gradient dynamics, the models presented here could be generalized to non-gradient systems, where birth-and-death rates and hopping rates do not satisfy the detailed-balance condition \cite{hoeksema2023thesis}, as in the case of \emph{active} agents. In these scenarios, the evolutionary Gamma-convergence approach used here is not directly applicable but has a potential for adaptation---this adaptation is currently under investigation.

\subsection{Approach and main results} We start by introducing the space $\X$ of finite configurations over a compact Polish space $X$. Elements $\x$ of $\X$ are symmetric $N$-tuples that can be identified with atomic measures $\nu\in \varGamma\coloneq\calM^+(X)$ with finite total variation, i.e.\ 
$$
    \X \coloneq \bigl\{ \x=(x^1,\ldots,x^N)\in X^N, ~ N\ge 1 \bigr\} / \sim,
$$
with $\sim$ being the equivalence relation $\sigma(\x) \sim \x$ for every permutation $\sigma$, then $\nu=\sum_{i=1}^N \delta_{x^i}$. On $\X$, one defines the so-called Lebesgue-Poisson measure $\lambda\in \varGamma$ (see \eqref{eq:Leb-Pois-measure} below for the defintion) w.r.t.\ an activity measure $\pi\in\varGamma$. Under an appropriate scaling corresponding to the large population limit, one obtains a family of scaled and normalized Lebesgue-Poisson measures $\lambda^n\in \calP(\X)$, $n\ge1$. Due to the identification $\x \leftrightarrow \nu$, one then obtains a \emph{reference measure} $\Pi^n\in \calP(\varGamma)$ that is supported on (scaled) atomic measures of the form $\nu=\frac{1}{n}\sum_{i=1}^N\delta_{x^i}$. Note that $N=N(t)$ is the actual number of particles, whereas $n\in\N$ is the scaling parameter corresponding to the average number of particles per unit volume according to $\pi$. Here, the space of finite measures $\varGamma=\calM^+(X)$ on $X$ is equipped with the narrow topology, i.e.\ in duality with continuous and bounded functions, making it a Polish space. The total variation norm for any finite measure $\nu\in\varGamma$ is simply denoted by $\|\nu\|$. We also consider the space of signed measures $\calM(Z)$ over various Polish spaces $Z$ and the weak-$*$ convergence, i.e.\ convergence against continuous functions that vanish at infinity. 

A large class of (scaled) IBMs for population dynamics can be modeled as $\varGamma$-valued stochastic process described via their corresponding (scaled) infinitesimal generator $L^n$. In our current context, $L^n$ is the sum of two parts consisting of the \emph{birth-and-death generator}
\[
    L_{\sfb\sfd}^n F(\nu) =  
    \int_X n\bigl[F(\nu + n^{-1}\delta_x)-F(\nu)\bigr]\,\sfb_n(\nu,\dd x) + \int_X n\bigl[F(\nu-n^{-1}\delta_x)-F(\nu)\bigr]\,\sfd_n(\nu,x)\,\nu(\dd x),
\]
with the (scaled) measure-dependent birth and death rate $\sfb_n$, $\sfd_n$, and the \emph{hopping (or Kawasaki) generator}
\[
    L_\sfh^n F(\nu) = \iint_{X\times X} n\bigl[F(\nu + n^{-1}(\delta_y-\delta_x))-F(\nu)\bigr]\,\sfh_n(\nu,x,\dd y)\,\nu(\dd x),
\]
with the (scaled) measure-dependent hopping rate $\sfh_n$. The time-marginal law $(\sfP_t^n)_{t\ge 0}$ of the $\varGamma$-valued process then satisfies the \emph{forward Kolmogorov equation}
\begin{align}\label{eq:FKEn}
	\partial_t\sfP_t^n = (L^n)^*\sfP_t^n \tag{\textsf{FKE}$_n$}.
\end{align}

\medskip

As mentioned, we will restrict ourselves to simple rates that satisfy the detailed balance condition:

\begin{assu}[Detailed balance and regularity]\label{ass:db}
Let $(\sfb_n,\sfd_n,\sfh_n)_{n\ge 1}$ be a sequence of measurable rates
\[
    \sfb_n\colon\varGamma\times\calB(X)\to[0,+\infty),\quad \sfd_n\colon\varGamma\times X\to[0,+\infty),\quad \sfh_n\colon \varGamma\times X\times\calB(X)\to[0,+\infty),
\]
satisfying the detailed balance condition:
\begin{align}\label{eq:DB-intro}\tag{\textsf{DB$_n$}}
    \left\{\quad\begin{aligned}
    \sfb_n(\nu,\dd x)&=\sfd_n(\nu+n^{-1}\delta_x,x)\,\pi(\dd x), \\
    \sfh_n(\nu,x,\dd y)\,\pi(\dd x) &= \sfh_n(\nu+n^{-1}(\delta_y-\delta_x),y,\dd x)\,\pi(\dd y),
    \end{aligned}\right.
\end{align}
and the uniform bounds
\[
    \sup_{n\ge 1}\sup_{(\nu,x)\in\varGamma\times X}\Bigl\{\|\sfb_n(\nu,\cdot)\| + \sfd_n(\nu,x) + \|\sfh_n(\nu,x,\cdot)\|\Bigr\} <+\infty,
\]
We further assume that there exist limit rates $(\sfb,\sfd,\sfh)$, such that
\[
    \sup_{(\nu,x)\in\varGamma\times X}\Bigl\{\|\sfb_n(\nu,\cdot)-\sfb(\nu,\cdot)\| +  |\sfd_n(\nu,x)-\sfd(\nu,x)| + \|\sfh_n(\nu,x,\cdot)-\sfh(\nu,x,\cdot)\|\Bigr\}\longrightarrow 0\qquad\text{as $n\to\infty$},
\]
where the limit rates automatically satisfy the limit detailed balance condition:
\begin{align}\label{eq:DB-limit}\tag{\textsf{DB$_\infty$}}
    \left\{\quad\begin{aligned}
    \sfb(\nu,\dd x)&=\sfd(\nu,x)\,\pi(\dd x), \\
    \sfh(\nu,x,\dd y)\,\pi(\dd x) &= \sfh(\nu,y,\dd x)\,\pi(\dd y).
    \end{aligned}\right.
\end{align}
For limit passage and the well-posedness of the limit equation, we will assume additional regularity of the limit rates. Namely, we assume 
\begin{enumerate}
    \item \emph{$\|\cdot\|$-Lipschitz regularity}: There exist Lipschitz constants $\ell_j$, $j\in\{\sfb,\sfd,\sfh\}$ such that
\begin{align*}
\left.\begin{gathered}
    \|\sfb(\nu) - \sfb(\eta)\| \le \ell_\sfb\|\nu-\eta\|,\qquad \sup_{x\in X}|\sfd(\nu,x) - \sfd(\eta,x)|\le \ell_\sfd\|\nu-\eta\| \\
    \sup_{x\in X}\|\sfh(\nu,x,\cdot) - \sfh(\eta,x,\cdot)\|\le \ell_\sfh\|\nu-\eta\|,
\end{gathered}\quad\right\}\quad\forall\,\nu,\eta\in\varGamma.
\end{align*}

\item \emph{Continuity}: For any $A\in X\times\calB(X)$ the functions
\[
    \nu\mapsto \sfb(\nu,A),\quad (\nu,x)\mapsto \sfd(\nu,x),\quad (\nu,x)\mapsto\sfh(\nu,x,A)
\]
are jointly narrowly continuous in $\nu\in\varGamma$ and continuous in $x\in X$.
\end{enumerate}
\end{assu}

\begin{example} An example of rates satisfying Assumption~\ref{ass:db} is given by
\begin{gather*}
    \sfb_n(\nu,\dd x) = \psi_{\sfb\sfd}\left(\int_{y\in X} c(x,y)\,\nu(\dd y)\right)\pi(\dd x) = \sfd_n(\nu,x)\pi(\dd x),\qquad c(x,x)=0, \\
    \sfh_n(\nu,x,\dd y) = \psi_\sfh\left(\int_{z\in X} h(x-z,y-z)\,\nu(\dd z)\right)\pi(\dd y),\qquad h(x-y,0) = h(0,y-x),
\end{gather*}
where $\psi_j\colon\R\to[0,\infty)$, $j\in\{\sfb\sfd,\sfh\}$ are Lipschitz functions and $c,h$ are bounded continuous functions.
\end{example}

Introducing the (scaled) jump kernel $\kappa^n:=\kappa^n_{\sfb\sfd} + \kappa^n_{\sf h}:\varGamma\times\calB(\varGamma)\to[0,\infty)$ with
\begin{gather*}
    \kappa_{\sfb\sfd}^n(\nu,\dd\eta) \coloneq\int_{X} \delta_{\nu+n^{-1}\delta_x}(\dd \eta)\,\sfb_n(\nu,\dd x) + \int_{X} \delta_{\nu-n^{-1}\delta_x}(\dd \eta)\,\sfd_n(\nu,x)\nu(\dd x),\\
      \kappa_{\sfh}^n(\nu,\dd\eta) \coloneq \iint_{X\times X} \delta_{\nu+n^{-1}(\delta_y-\delta_x)}(\dd \eta)\,\sfh_n(\nu,x,\dd y)\nu(\dd x),
\end{gather*}
the generator $L^n$ may then be expressed as
\[
	L^nF(\nu) = \int_{\varGamma} \dnabla_n F(\nu,\eta)\,\kappa^n(\nu,\dd\eta),
\]
where $\dnabla_n$ is the (scaled) \emph{discrete gradient} defined by $\dnabla_n F(\nu,\eta) = n[F(\nu)-F(\eta)]$. The previous expression shows that $L^n$ is the generator of a pure jump process on the space of finite measures $\varGamma$.

In Section~\ref{sec:particle-dynamics}, we show that under Assumption~\ref{ass:db}, the jump kernel $\kappa^n$ satisfies the usual detailed-balance condition for jump processes w.r.t.~the reference measure $\Pi^n$, i.e., the measure 
\[
	\kappa^n(\nu,\dd\eta)\,\Pi^n(\dd\nu)\quad\text{is symmetric}.
\]
Following the approach in \cite{PRST2022}, one obtains the well-posedness and a generalized gradient flow formulation for the forward Kolmogorov equation \eqref{eq:FKEn} with a generalized gradient structure $(\sfR_n,\sfR_n^*,\sfE_n)$, where $\sfR_n,\sfR_n^*$ are dissipation and dual dissipation potentials, respectively, and $\sfE_n$ is the driving energy for the gradient flow, which is taken to be the (scaled) relative entropy $\sfE_n\coloneq n^{-1}\Ent(\cdot|\Pi^n)\colon \calP(\varGamma)\to[0,+\infty]$, where
\[
    \Ent(\alpha|\beta)\coloneq \begin{cases} \displaystyle\int \upphi\biggl(\frac{\dd\alpha}{\dd\beta}\biggr)\dd\beta & \text{if $\alpha\ll \beta$}, \\
    +\infty &\text{otherwise},
    \end{cases}\qquad \upphi(s)=s\log s-s+1,\;\; s\ge 0.
\]
In this work, we will be adopting the so-called `cosh' generalized gradient structure (cf.\ Section~\ref{ss:nggf}), as they are the \emph{natural} ones from the large-deviation perspective (see e.g.\ \cite{peletier2023} for more insights into this structure).

Formally, this means that \eqref{eq:FKEn} may be written as
\begin{align}
	\partial_t\sfP_t + \ddiv_n\, \sfJ_t &= 0 \tag{\textsf{CE}$_n$}\label{eq:CE}\\ 
	\sfJ_t &= \partial_2\sfR_n^*(\sfP_t,-\dnabla_n \sfE_n'(\sfP_t)) \tag{\textsf{KE}$_n$}\label{eq:KE},
\end{align}
where $\ddiv_n$ denotes the (scaled) \emph{discrete divergence}, \eqref{eq:CE} is the \emph{continuity equation} for the density-flux pair $(\sfP,\sfJ)$, and \eqref{eq:KE} is the \emph{kinetic relation}, which provides the relationship between the force $\sfE_n'(\sfP)$ induced by the energy $\sfE_n$ at $\sfP$ and the flux $\sfJ$ via the dual dissipation potential $\sfR_n^*$. By \emph{Legendre duality}, one has that
\begin{align}\label{eq:LD}
	\sfR_n(\sfP,\sfJ) + \sfR_n^*(\sfP,\upzeta) \ge \langle\sfJ,\upzeta\rangle\qquad\forall\, (\sfJ,\upzeta)\tag{\textsf{LD}},
\end{align}
with equality if and only if $\sfJ = \partial_2\sfR_n^*(\sfP,\upzeta)$. Hence, the \eqref{eq:KE} may be equivalently expressed as
\begin{align}\label{eq:KE2}
	\sfR_n(\sfP_t,\sfJ_t) + \sfR_n^*(\sfP_t,-\dnabla \sfE_n'(\sfP_t)) = \langle \sfJ_t,-\dnabla_n \sfE_n'(\sfP_t)\rangle.\tag{\textsf{KE}$_n'$}
\end{align}
On the other hand, an application of the chain rule gives
\begin{align}\label{eq:CR}
	\frac{\dd}{\dd t}\sfE_n(\sfP_t) = \langle \partial_t\sfP_t,\sfE_n(\sfP_t)\rangle = \langle \sfJ_t,\dnabla_n \sfE_n'(\sfP_t)\rangle,\tag{\textsf{CR}$_n$}
\end{align}
where the second equality follows from the continuity equation \eqref{eq:CE}. Combining the expressions \eqref{eq:KE2} and \eqref{eq:CR}, and integrating over arbitrary intervals $[s,t]\subset[0,T]$, one obtains the \emph{energy-dissipation balance}
\begin{align}\label{eq:EDB}
	\mathscr{I}_n([s,t];\sfP,\sfJ)\coloneq \int_s^t \sfR_n(\sfP_r,\sfJ_r) + \sfR_n^*(\sfP_r,-\dnabla_n \sfE_n'(\sfP_r))\,\dd r + \sfE_n(\sfP_t) - \sfE_n(\sfP_s) = 0 \tag{\textsf{EDB}$_n$}.
\end{align}
In this way, the solution $\sfP^n$ of \eqref{eq:FKEn} can be associated with a density-flux pair $(\sfP^n,\sfJ^n)$ satisfying both the continuity equation \eqref{eq:CE} and energy-dissipation balance \eqref{eq:EDB}. Moreover, since $\mathscr{I}_n([s,t];\sfP,\sfJ)\ge 0$ for any density-flux pairs $(\sfP,\sfJ)$ satisfying \eqref{eq:CE}, the associated pair $(\sfP^n,\sfJ^n)$ is in fact a minimizer of the \emph{energy-dissipation functional} $\mathscr{I}_n$, thus providing a variational characterization for the solution of \eqref{eq:FKEn}.

Section~\ref{sec:limit} deals with the large population limit, where we show that a subsequence of generalized gradient flow solutions $(\sfP^n,\sfJ^n)_{n\ge 1}$ to \eqref{eq:FKEn} converges to some limit pair $(\sfP,\sfJ)$ in an appropriate sense as $n\to\infty$, where the limit pair $(\sfP,\sfJ)$ satisfies 
\begin{enumerate}[label=(\textit{\roman*})]
    \item the limit continuity equation
\begin{align}\label{eq:CElimit}
	\partial_t\sfP_t + \text{div}_{\varGamma}\, \sfJ_t = 0,\tag{\textsf{CE$_\infty$}}
\end{align}
where $\text{div}_{\varGamma}$ is the $\varGamma$-divergence operator associated to the $\varGamma$-gradient operator
\[
	\nabla_{\varGamma} F(\nu,x) \coloneq \sum_{m=1}^M \partial_m g\bigl(\langle f_1,\nu\rangle,\ldots,\langle f_m,\nu\rangle\bigr)\, f_m(x),\qquad (\nu,x)\in\varGamma\times X,
\]
defined for cylindrical functions $F(\nu) = g(\langle f_1,\nu\rangle,\ldots,\langle f_M,\nu\rangle)$ with $g\in C_c^\infty(\R^M)$ and $f_m\in C_b(X)$,

\item and the limit energy-dissipation balance
\begin{align}\label{eq:EDBlimit}
	\mathscr{I}_\infty([s,t];\sfP,\sfJ)\coloneq \int_s^t \sfR_\infty(\sfP_r,\sfJ_r) + \sfR_\infty^*(\sfP_r,-\nabla_\varGamma \sfE_\infty'(\sfP_r))\,\dd r + \sfE_\infty(\sfP_t) - \sfE_\infty(\sfP_s) = 0 \tag{\textsf{EDB}$_\infty$},
\end{align}
where $(\sfR_\infty,\sfR_\infty^*,\sfE_\infty)$ is the limit generalized gradient structure with driving energy
\[
    \sfE_\infty(\sfP) = \int_{\varGamma} \Ent(\nu|\pi)\,\sfP(\dd\nu).
\]
\end{enumerate}
The above convergence is established using evolutionary Gamma-convergence techniques, further resulting in \emph{entropic propagation of chaos}, which is an important feature of our approach.

In particular, we show that the limit density-flux pair $(\sfP,\sfJ)$ is a generalized gradient flow solution of the so-called \emph{Liouville equation}
\begin{align}\label{eq:LiE}\tag{\textsf{LiE}}
    \partial_t\sfP_t + \text{div}_\varGamma(\scrV\sfP_t) = 0,
\end{align}
where $\scrV\coloneq \kappa_{\sfb\sfd} + \kappa_\sfh:\varGamma\to\varGamma$ is the measure-valued vector field associated to the limit jump kernels 
\begin{align}\label{eq:rates-limit}
    \begin{aligned}
    \kappa_{\sfb\sfd}[\nu](\dd x) &\coloneq \sfb(\nu,\dd x) - \sfd(\nu,x)\nu(\dd x), \\
    \kappa_\sfh[\nu](\dd x) &\coloneq \int_{X} \sfh (\nu, y,\dd x)\,\nu(\dd y) - \nu(\dd x) \int_{X} \sfh (\nu,x,\dd y).
    \end{aligned}
\end{align}
Observe that the Liouville equation \eqref{eq:LiE} resembles a standard transport equation with the velocity field $\scrV$, whose solution can be obtained by studying its \emph{Lagrangian flow} $\scrS:[0,T]\times \varGamma\to\varGamma$, associated to 
\begin{align}\label{eq:LF}
    \partial_t \nu_t = \scrV[\nu_t],\qquad \nu\in\varGamma.\tag{\textsf{MFE}}
\end{align}
This is the \emph{mean-field equation}. Assuming the well-posedness of \eqref{eq:LF}, the unique solution of \eqref{eq:LiE} is then given by $\sfP_t\coloneq (\scrS_t)_\#\sfP_0$, which can be established utilizing a \emph{superposition principle} (cf.\ Theorem~\ref{th:superposition}).  

In the special case $\sfP_0=\delta_{\nu_0}$ with $\Ent(\nu_0|\pi)<+\infty$, the well-posedness of \eqref{eq:LF} yields the solution $\sfP_t = \delta_{\nu_t}$ with $\nu_t\coloneq \scrS_t(\nu_0)$ and $\Ent(\nu_t|\pi)<+\infty$ for every $t\in[0,T]$. Hence, setting $u_t\coloneq \dd\nu_t/\dd\pi$, one obtains \eqref{eq:mf-intro} from \eqref{eq:LF} with the density-dependent functions
\begin{align}\label{eq:rates-density}
    \sfr_\sfb(u,x) =  \frac{\dd \sfb(u\pi,\cdot)}{\dd\pi}(x),\qquad \sfr_\sfd(u,x) = \sfd(u\pi,x),\qquad \sfr_\sfh(u,x,y) = \frac{\dd\sfh(u\pi,y,\cdot)}{\dd\pi}(x),
\end{align}
which are well-defined due to the detailed balance conditions in Assumption~\ref{ass:db}. As a by-product of the energy-dissipation balance \eqref{eq:EDBlimit}, we obtain a generalized gradient structure $(\scrR,\scrR^*,\scrE)$ for the mean-field equation \eqref{eq:LF} with the driving energy $\scrE = \Ent(\cdot|\pi)\colon\varGamma\to[0,+\infty]$.

\begin{figure}[ht]
\centering
\begin{tikzcd}
\begin{array}{c}
\textbf{Forward Kolmogorov equation
} \\[0.2em]
    \eqref{eq:FKEn}
\end{array}\qquad
\arrow[r, "n\to\infty", start anchor={[xshift=-4ex]}, end anchor={[xshift=5ex]}] 
& 
\qquad \begin{array}{c}
\textbf{Liouville equation} \\[0.2em]
\eqref{eq:Liouville}
\end{array} 
\arrow[d, leftharpoonup, start anchor={[xshift=2.5ex]}, end anchor={[xshift=2.5ex]}]
\arrow[d, rightharpoondown, start anchor={[xshift=1.5ex]}, end anchor={[xshift=1.5ex]}]
\\
\text{}
&
\qquad\begin{array}{c}
    \textbf{Mean-field equation} \\[0.2em]
     \eqref{eq:MF}
\end{array}
\end{tikzcd}
\caption{A depiction of our approach to proving the large-population limit.} \label{fig:convergence}
\end{figure}
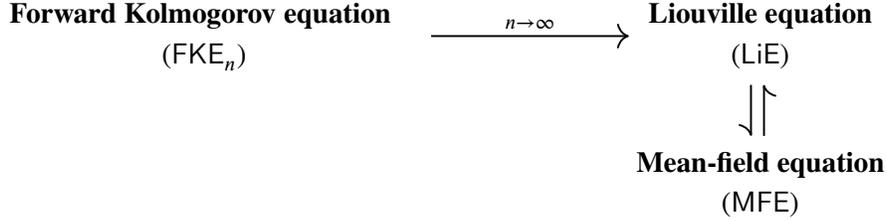 

The discussion above can be summarized in the following result, which is a consequence of Theorem~\ref{th:convergence}.

\begin{mainthm}\label{th:main}
Let $\{(\sfP^n, \sfJ^n)\}_{n\in\N}$ be a sequence of generalized gradient flow solutions to \eqref{eq:FKEn} with initial data $\{\sfP^n_\text{in}\}_{n\in\N}$ and suppose there exists $\sfP_\text{in}\in \text{dom} (\sfE_\infty)$ such that
$$
    \sfP^n_\text{in} \rightharpoonup \sfP_\text{in},\qquad \lim_{n\to\infty} \sfE_n (\sfP^n_\text{in}) = \sfE_\infty(\sfP_\text{in}).
$$
Then there exists a unique generalized gradient flow solution $(\sfP,\sfJ)$ of \eqref{eq:Liouville} with $\sfP_0=\sfP_\text{in}$, such that
\begin{itemize}
         \item $\sfP^n_t \rightharpoonup \sfP_t$ narrowly for all $t\in [0,T]$;
 \item $\lim_{n\to\infty} \sfE_n (\sfP^n_t) = \sfE_\infty(\sfP_t)$ for all $t\in [0,T]$.
\end{itemize}

\medskip

If $\sfP_\text{in}=\delta_{\bar\nu}$, $\bar\nu\in\text{dom}(\scrE)$, then $\sfP_t=\delta_{\nu_t}$, where $\nu$ solves \eqref{eq:LF} and satisfies the energy-dissipation balance 
\[
    \scrI([s,t],\nu,\lambda) \coloneq \int_s^t \calR(\nu_r,\lambda_r) + \calR^*(\nu_r,-\calE'(\nu_t))\,\dd r + \calE(\nu_t) - \calE(\nu_s) = 0,
\]
with an associated flux $\lambda$. If in addition, $u_t\coloneq \dd \nu_t/\dd \pi\in[a,b]$, $0<a\le b\in [0,+\infty)$ for every $t\in[0,T]$, then the entropic propagation of chaos holds, i.e.
\begin{equation}\label{eq:mprop}
     \lim_{n\to\infty} \frac{1}{n}\Ent(\sfP_t^n|\Pi_{\nu_t}^n) = 0\qquad\text{for every $t\in[0,T]$},
\end{equation}
where $\Pi_\nu^n$ is the scaled Poisson measure with activity measure $n\nu$ (see Section~\ref{subsec:large-population} for details).

Finally, the curve $t\mapsto u_t$ solves the doubly nonlocal Fisher-KPP equation \eqref{eq:mf-intro}.  
\end{mainthm}

\begin{remark}[Entropic propagation of chaos]
In the canonical ensemble setting, i.e., when the number of particles $n\in\N$ is fixed and goes to infinity, with the distribution of the particles given by $\rho_t^n$, there are various notions of propagation of chaos \cite{chaintrondiez2022}. In this context, the weak convergence of $\sfP_t^n\to \delta_{\nu_t}$ (where $\nu_t$ is now a probability measure) is equivalent to the weak notion of chaoticity (e.g.\ see \cite{Sznitman1991}). This means that for every $K\in \N$, the $K$-marginal distributions $\rho^{n,K}_t$ of $\rho^n_t$ converge to $\nu_t^{\otimes K}$. Moreover, global entropic chaoticity, expressed by the convergence 
\[ 
    \frac{1}{n} \Ent(\rho^n_t|\nu_t^{\otimes n})\to 0\qquad\text{as $n\to\infty$,}
\]
directly implies, through the subadditivity of entropy, the local estimate convergence 
\[
    \Ent(\rho^{n,K}_t|\nu_t^{\otimes K}) \to 0\qquad\text{as $n\to\infty$.}
\]
In the grand canonical setting, the situation becomes more complex. While for certain interacting particle systems birth-and-death processes, the rescaled $K$-correlation measures do converge to $\nu_t^{\otimes K}$ under suitable strong moment assumptions \cite{Finkelshtein2010}, this does not extend to our setting where we only assume convergence of the rescaled entropies for the initial data. Moreover, for $K>1$, the assumption of finite entropy relative to $\Pi^n$ does not even ensure that the $K$-correlation measures are well-defined. Nevertheless, we propose the convergence \eqref{eq:mprop} as a suitable generalization of the global entropic propagation of chaos from the canonical to the grand canonical setting. This convergence suggests that the rescaled entropy of the law of the interacting particle system, relative to suitable Poisson measures, vanishes as $n\to \infty$. It is plausible that this would imply the convergence of appropriately weighted correlation measures, but we leave this to future research.
\end{remark}

\subsubsection*{Acknowledgements} The authors acknowledge support from NWO Vidi grant 016.Vidi.189.102 on ``Dynamical-Variational Transport Costs and Application to Variational Evolution''. A.H. acknowledges support from ERC Advanced Grant on ``Everything You Always Wanted to Know About the JKO Scheme.''

\section{Preliminaries: configuration spaces and large-population scaling}\label{sec:configurations}

\subsection{Configuration space}

The basis for the mathematical description of the IBMs we will consider relies on the framework of configuration spaces, which in our present context, can be found in the works by Kondratiev, Finkelshtein, and others \cite{finkelshtein2009individual, finkelshtein2010vlasov, finkelshtein2013markov}. Throughout, let $X$ be a compact Polish space. 

We define the family of \emph{$N$-point configurations} as the set of $N$-tuples of indistinguishable particles
	\begin{align*}
    \X^N \coloneq X^N\big/ \sim\,,\qquad N\in\N_0\coloneq \N\cup\{0\},\qquad \X^0\coloneq\emptyset,
\end{align*}
where $\sim$ is the equivalence relation $(x^1,\ldots, x^N) \sim (x^{\sigma_1},\ldots, x^{\sigma_N})$ for any permutation $\sigma$. \emph{The space of finite configuration} is then defined as
$$
    \X \coloneq \bigsqcup_{N \in \N_0} \X^N.
$$
We equip $\X$ with the disjoint union topology, which is the finest topology for which all the canonical injections $\varphi_N \colon \X^N\to \X$ are continuous. We also introduce the Borel $\sigma$-algebra on $\X^N$ denoted by $\calB(\X^N)$ and the corresponding Borel $\sigma$-algebra $\calB(\X)$ on $\X$. The continuous map $\sfN:\X\to\N_0$, which assigns $\sfN(\x)=N$ for every $N$-point configuration $\x\in \X^N$ is called the \emph{number map}.

\emph{The Lebesgue-Poisson measure} $\lambda$ on $\X$ with activity measure $\pi$ is defined as
\begin{equation}\label{eq:Leb-Pois-measure}
    \lambda \coloneq \sum_{N=0}^\infty \frac{1}{N!} \pi^{\otimes N}\in\calM^+(\X),
\end{equation}
where $\pi\in\varGamma$ is a given finite Radon measure on $X$. For example, if $X$ is a compact subset of $\R^d$, one can choose $\pi$ to be the restriction of the Lebesgue measure to $X$. It is easy to see that if $\pi$ is a finite measure, then the Lebesgue-Poisson measure $\lambda$ is also finite:
$$
    \lambda(\X) = \sum_{N=0}^\infty \frac{1}{N!} \|\pi\|^N = e^{\|\pi\|} < \infty.
$$
Together, the triplet $(\X, \calB(\X), \lambda)$ forms a measure space.

\begin{remark} 
	Let us discuss some modifications of the configuration space $\X$ that are relevant for particular mathematical models.

The case when particles have multiple types is relevant for biological models such as tumor growth models, where tissue cells can be either healthy,  proliferative,
non-proliferative, or necrotic  \cite{kansal2000simulated}. Other examples include the classical SIR model in epidemiology that includes susceptible, infected, and recovered agents \cite{kermack1927contribution}, and ecological models of various competing species of plants or animals \cite{fournier2004microscopic}. Supposing that the system has $K \in \N$ different types of particles and that each particle can change its type, the generalization of $\X^n$ to multi-types can be described with a multi-index $\n=(n_1,\ldots,n_K) \in \N_0^K$, $n_l\in\N_0$, by setting
\begin{align*}
\Y^0 \coloneq \emptyset,\qquad \Y^{\n} \coloneq \X^{n_1} \times \cdots \times \X^{n_K},
\end{align*}
and defining the space of finite configurations as
\begin{align*}
    \Y \coloneq \bigsqcup_{\n \in \N_0^K} \Y^{\n}.
\end{align*}
Similarly to the one-type case, this space is endowed with the disjoint union topology and the corresponding number map $\sfN:\Y\to\N_0$, which assigns $\sfN(\boldsymbol{\x})=n_1+\cdots+n_K$ for every $\boldsymbol{\x}\in \Y^{\n}$.

An important restriction that we need to include in the definition of the configuration space is that each particle can have only one type. Therefore, we consider a subset $\Y_{\neq} \subset \Y$ defined as
\[
    \Y_{\neq} \coloneq \biggl\{\boldsymbol{\x}\in\Y\;\biggl|\; \boldsymbol{\x}=(\x_1,\ldots,\x_K),\;\; x_l^i \ne x_m^j\;\;\forall\,l,m\in[K],\;l\ne m,\; i\in[n_l],\;\;j\in [n_m]\biggr\},
\]
where $[n]=\{1,\ldots,n\}$, $n\in\N$.

Here, we have the choice to either restrict the configuration space to $\Y_{\neq}$ or to define a measure on $\Y$ having its support contained in $\Y_{\neq}$---we choose the latter. Following the single-type particle configuration, we define the Lebesgue-Poisson measure $\lambda$ on $\Y$ as
\[
    \lambda \coloneq \sum_{n_1, \ldots, n_K =0}^\infty \frac{1}{\prod_{p=1}^K n_p!} \pi^{\otimes n_1} \otimes \dots \otimes \pi^{\otimes n_K}.
\]
We then set $\lambda_{\neq} \coloneq \mathbf{1}_{\Y_\ne}\lambda$, meaning that $\lambda_{\neq}$ is the restriction of $\lambda$ to $\Y_{\neq}$.

	A more general approach to introducing diversity in the particle types is suggested, for instance, in \cite{kondratiev1999analysis} using \emph{marked configuration spaces}. The marked configuration space $\Omega^M_X$ over $X$ with marks from a set $M$ is defined as
\begin{equation*}
    \Omega^M_X \coloneq \{(\x, m) : \x \in \X, ~ m\in M^{\x}\}.
\end{equation*}
where $M^{\x}$ denotes the set of all maps $[\sfN(\x)]\ni i\mapsto m_i \in M$, which is permutation invariant.
\end{remark}

\subsection{Large-population scaling}\label{subsec:large-population}
In performing the large population limit, we consider an appropriate scaling and normalization of the Lebesgue-Poisson measure, i.e.\ we set for each $n\ge 1$,
\[
	\lambda^n\coloneq e^{-n\|\pi\|}\sum_{N=0}^\infty \frac{n^N}{N!} \pi^{\otimes N}\in\calP(\X),
\]
i.e.\ $\lambda^n$ is a normalization of the Lebesgue-Poisson measure with activity measure $n\pi$.

In this way, the expected population size scales as $n$ since
\[
	\int_{\X} \sfN(\x)\,\lambda^n(\dd\x) = e^{-n\|\pi\|} \sum_{N=0}^\infty \frac{N n^N}{N!} \|\pi\|^N = n\|\pi\|.
\]
We further introduce the (scaled) atomic measure map
\[
	\Lambda^n:\X\to \varGamma,\qquad \x\mapsto\Lambda^n(\x) \coloneq \frac{1}{n}\sum_{i=1}^N \delta_{x^i},\quad \x=(x^1,\ldots,x^N)\in\X^N,
\]
which is easily verified to be measurable for each $n\ge 1$. One can then push the scaled Lebesgue-Poisson measure $\lambda^n$ on $\X$ along the map $\Lambda^n$ to a probability measure $\Pi^n$ on $\varGamma$, which serves as the reference measure on $\varGamma$---we call $\Pi^n$ the scaled \emph{Poisson measure}. Then by construction, the reference measure $\Pi^n$ satisfies
\begin{align*}
	\int_{\varGamma} \|\nu\|\,\Pi^n(\dd\nu) = \frac{1}{n}\int_{\X} \sfN(\x)\,\lambda^n(\dd \x) = \|\pi\|\qquad\forall\, n\ge 1,
\end{align*}
ensuring that the expected mass of scaled atomic measures $\nu$ sampled from $\Pi^n$ equals the mass of the activity measure $\pi$. Clearly, the previous equality also provides a uniform finite first-moment estimate for the family $\{\Pi^n\}_{n\ge 1}$. Notice that the support of any measure $\sfP\in\calP(\varGamma)$ satisfying $\sfP\ll \Pi^n$ contains only atomic measures. Moreover, the reference measure $\Pi^n$ has exponential moments with
\begin{align*}
    \frac{1}{n}\log\int_\varGamma e^{n\|\nu\|}\,\Pi^n(\dd\nu) = (e-1)\|\pi\|\qquad\forall\,n\ge 1.
\end{align*}
In particular, if $\sfP\in\calP(\varGamma)$ has finite energy, i.e.\ $\sfE_n(\sfP)<+\infty$, then $\sfP$ has finite first moment, which follows from a variational representation of exponential integrals (cf.\ \cite[Lemma B.1]{hoeksema2020}), generalizing Varadhan's representation of the relative entropy \cite{varadhan1984}. More explicitly, for the scaled driving energy $\sfE_n = n^{-1}\Ent(\cdot|\Pi^n)$, one has that
\begin{align}\label{eq:moment}
    \int_\varGamma \|\nu\|\,\sfP(\dd\nu) \le \sfE_n(\sfP) + \frac{1}{n}\log \int_{\varGamma} e^{n|\nu|}\,\Pi^n(\dd \nu) = \sfE_n(\sfP) + (e-1)\|\pi\|.
\end{align}
These facts play an essential role in the asymptotic limit of $\sfE_n$, which we now illustrate to provide a first glimpse into our approach.

Recall the functional $\sfE_\infty\colon \varGamma\to [0,+\infty]$ mentioned in the introduction, given by
\[
	\sfE_\infty(\sfP)\coloneq 	\int_{\varGamma} \Ent(\nu|\pi)\,\sfP(\dd\nu) = \int_{\varGamma} \scrE(\nu)\,\sfP(\dd\nu).
\]

\begin{proposition}\label{prop:energy_lim}
	Let $(\sfP^n)_{n\ge 1}\subset\calP(\varGamma)$ be a sequence satisfying
	\[
		\sup\nolimits_{n\ge 1} \sfE_n(\sfP^n)<+\infty.
	\]
	Then, there exists some $\sfP\in\calP(\varGamma)$ such that $\sfP^n\rightharpoonup\sfP$ narrowly in $\calP(\varGamma)$. In addition,
	\[
		\sfE_\infty(\sfP) \le \liminf_{n\to\infty} \sfE_n(\sfP^n).
	\]
\end{proposition}

The proof of Proposition~\ref{prop:energy_lim} is rather elementary but requires a delicate approximation result that can be found in \cite{Mariani2012}. We include the proof here to highlight the general strategy used to establish the corresponding limits for the other terms in the energy-dissipation functional $\scrI_n$.
\begin{proof}[Proof of Proposition~\ref{prop:energy_lim}]
	For any $n\ge 1$, one uses inequality \eqref{eq:moment} to obtain
\begin{align*}
	\int_{\varGamma} |\nu|\,\sfP^n(\dd\nu) \le \sup_{n\ge 1}\sfE_n(\sfP^n) + (e-1)\|\pi\|,
\end{align*}
therewith implying the tightness of the sequence $(\sfP^n)_{n\ge 1}\subset\calP(\varGamma)$ w.r.t.\ narrow convergence. One can then extract a (not relabelled) subsequence and a point $\sfP\in\calP(\varGamma)$ for which $\sfP^n\rightharpoonup\sfP$ narrowly in $\calP(\varGamma)$.

Suppose for the moment that $\sfP=\delta_{\eta}$ for some $\eta\in\varGamma$. For cylinder functions of the form $F(\nu)=\langle \varphi,\nu\rangle$, $\varphi\in B_b(X)$, Varadhan's representation \cite{varadhan1984} then yields
\begin{align*}
	\liminf_{n\to\infty}\sfE_n(\sfP^n) &\ge \langle \varphi,\eta\rangle - \int_X \bigl(e^{\varphi(x)} - 1\bigr)\,\pi(\dd x)
\end{align*}
Taking the supremum over $\varphi\in B_b(X)$ and using Varadhan's representation again, we arrive at
\[
	\liminf_{n\to\infty} \sfE_n(\sfP^n) \ge \Ent(\eta|\pi) = \int_\varGamma \Ent(\nu|\pi)\,\sfP(\dd\nu).
\]
The result for general $\sfP\in\calP(\varGamma)$ follows from \cite[Theorem~3.5]{Mariani2012}.
\end{proof}

\begin{remark}
    In general, one could also consider reference measures in Gibbs form, i.e.\ 
    \[
        e^{-\calF_n(\nu)}\,\Pi^n(\dd\nu),\qquad \calF_n\colon\varGamma\to(-\infty,+\infty].
    \]
    The detailed balance condition on the jump kernel $\kappa^n$ is then specified relative to this measure.
\end{remark}

\section{The forward Kolomogorov equation}\label{sec:particle-dynamics}

We now turn to the forward Kolmogorov equation \eqref{eq:FKEn}, which describes the law of the measure-valued process corresponding to the IBM. We will discuss well-posedness, the phrasing of \eqref{eq:FKEn} in terms of a continuity equation, and its variational formulation via the generalized gradient structure outlined previously. 

\subsection{Continuity equation}\label{sec:fke-con}

Let $T>0$ be an arbitrary but fixed time horizon. We consider curves $(\sfP_t^n)_{t\ge 0} \subset \calP(\varGamma)$ that satisfy the continuity equation \eqref{eq:CE} in the following sense:

\begin{definition}\label{def:CEn}
A tuple $(\sfP,\sfJ^{\sfb\sfd},\sfJ^\sfh)$ satisfies the continuity equation, or simply $(\sfP,\sfJ)\in \mathsf{CE}_n$, if 
\begin{enumerate}
    \item \label{cont_i4} The map $[0,T]\ni t\mapsto \sfP_t \in \calP(\varGamma)$ is narrowly continuous, and
    \[
        \sup_{t\in [0,T]} \int_{\varGamma} \|\nu\| \dd \sfP_t  < \infty.
    \]
    \item \label{cont_i2} The Borel families $\{\sfJ^{\sfb\sfd}_t\}_{t\geq 0} \subset \calM(\varGamma\times X)$, $\{\sfJ^{\sfh}_t\}_{t\geq 0} \subset \calM(\varGamma\times X\times X)$ satisfy 
    \[
        \mathrm{supp}\,\sfJ_t^{\sfh}\subset \bigl\{ (\nu,x,y) \, : x\in \mathrm{supp}\,\nu \bigr\},\qquad \int_0^T \|\sfJ_t^{\sfb\sfd}\|+\|\sfJ_t^{\sfh}\|\, \dd t <\infty.
    \]
    \item \label{cont_i3} For any bounded Borel function $F\in B_b(\varGamma)$ and interval $[s,t]\subset [0,T]$,
    \[\int_{\varGamma} F \dd \sfP_t - \int_\varGamma F\,\dd \sfP_s= \int_s^t \left[ \iint_{\varGamma\times X}\dnabla_{n,\sfb\sfd}F(\nu,x)\,\sfJ^{\mathsf{bd}}_r(\dd \nu\dd x)+ \iiint_{\varGamma\times X^2}\dnabla_{n,\sfh}F(\nu,x,y)\,\sfJ^{\sfh}_r(\dd \nu\dd x\dd y) \right] \dd r,
    \]
    with (scaled) discrete derivatives 
    \begin{equation*}
    \dnabla_{n,\sfb\sfd} F(\nu,x)\coloneq n\bigl[F(\nu+n^{-1}\delta_x)-F(\nu)\bigr], \qquad \dnabla_{n,\sfh}F(\nu,x,y) \coloneq n\bigl[ F(\nu+n^{-1}(\delta_y-\delta_x))-F(\nu)\bigr].
\end{equation*}
\end{enumerate}
\end{definition}

Formally the continuity equation can be viewed in terms of a discrete divergence in the sense that
\[
	\partial_t\sfP_t + \ddiv_{n,\sfb\sfd}\, \sfJ_t^{\sfb\sfd} + \ddiv_{n,\sfh}\, \sfJ_t^{\sfh} = 0,
\]
where $\ddiv_{n,\sfb\sfd}$ and $\ddiv_{n,\sfh}$ are the dual to the discrete derivatives $\dnabla_{n,\sfb\sfd}$ and $\dnabla_{n,\sfh}$ respectively.
  
\begin{remark}
Condition \eqref{cont_i2} of Definition~\ref{def:CEn} is satisfied whenever $\sfJ_t^{\sfh}(\dd \nu,\dd x,\dd y)= \sfh_{t}(\nu,x,\dd y)\nu(\dd x)\sfP_t(\dd \nu)$ for some time-and-measure dependent kernel $\sfh$. Moreover, the moment Condition \eqref{cont_i4} holds by using similar arguments from \cite{PRST2022}, since eventhough the kernel $\kappa^n$ is unbounded, we have 
\begin{equation}\label{eq:remkb}
    \sup_{t\in [0,T]} \int_{\varGamma} \kappa^n(\nu,\varGamma)\,\sfP(\dd \nu) <\infty. 
\end{equation}
\end{remark}

\begin{remark}\label{rem:fke_edgeflux}
The above form of the continuity equation is suitable for taking large population limits, but to apply the tools of \cite{PRST2022} it is useful to bring this into a form of the continuity equation \eqref{eq:CE} similar to the one stated in the introduction. In particular, with a little abuse of notation, setting the \emph{edge flux}
\begin{align*}
        \sfJ^\circ_t(\dd \nu\dd \eta)\coloneq & \frac{n}{2} \int_{x\in X} \delta_{\nu+n^{-1}\delta_x}(\dd \eta)\, \sfJ_t^{\sfb\sfd}(\dd \nu,\dd x)-\frac{n}{2} \int_{x\in X} \delta_{\nu-n^{-1}\delta_x}(\dd \eta)\, \sfJ_t^{\sfb\sfd}(\dd (\nu-\delta_x),\dd x) \\
        &\qquad + n \iint_{(x,y)\in X^2} \delta_{\nu+ n^{-1}(\delta_y-\delta_x)}(\dd \eta)\, \sfJ_t^{\sfh}(\dd \nu,\dd x\dd y),
\end{align*}

Conditions~\eqref{cont_i2} and \eqref{cont_i3} of Definition~\ref{def:CEn} can then be rephrased as
\begin{enumerate}
    \item[(2')] The Borel family $\{\sfJ^\circ_t\}_{t\geq 0}\subset \calM(\varGamma\times \varGamma)$ has finite total variation over $[0,T]$.
    \item[(3')] For any bounded Borel function $F\in B_b(\varGamma)$ and interval $[s,t]\subset [0,T]$, we have
    \[
    \int_{\varGamma} F\dd\sfP_t- \int_\varGamma F\dd \sfP_s = \int_s^t \iint_{\varGamma^2} (F(\nu)-F(\eta)) \,\sfJ^\circ_r(\dd\nu\dd\eta) \dd r.
\]
\end{enumerate}
Vice versa, any Borel family $\{\sfJ^\circ_t\}_{t\ge 0}$ such that $\sfJ^\circ_t(\dd \nu\dd \eta)\ll \kappa^n(\nu,\dd \eta)\Pi^n(\dd \nu)$ can be decomposed into a birth-and-death $\sfJ_t^{\sfb\sfd}$ and a hopping $\sfJ_t^\sfh$ part. 
\end{remark}

In Lemma~\ref{lm:rever}, we show the reversibility of the kernel $\kappa^n$ with respect to $\Pi^n$. Together with the above this implies, via a truncation argument of the kernels $\kappa^n$, mass moment estimates on $\sfP_t$, and adaptation from the techniques in \cite{PRST2022}, that whenever $\sfE_n(\sfP_0)<\infty$ with suitable moment estimates, classical solutions to \eqref{eq:FKEn} exist and have a variational representation in terms of generalized gradient flows, see Theorem \ref{th:fke_main}.

\subsection{Solutions to the forward Kolmogorov equation}\label{sec:fke-sol}

We consider the following family of solutions to the forward Kolmogorov equation \eqref{eq:FKEn}.

\begin{definition}\label{def:sol-FKEn}
A narrowly continuous curve $\sfP = (\sfP_t)_{t\in [0, T]}$ is a solution to \eqref{eq:FKEn} if
\begin{equation}\label{eq:sol_b1}
 \sup_{t\in [0,T]} \int_{\varGamma} \|\nu\| \dd \sfP_t  < \infty,
\end{equation}
and for every bounded Borel function $F\in B_b(\varGamma)$ and interval $[s,t]\subset [0,T]$,
\begin{equation}\label{eq:sol_c}
    \int_{\varGamma} F \dd \sfP_t - \int_{\varGamma} F \dd \sfP_s = \int_s^t \int_{\varGamma} L^n F (\nu) \, \sfP_r(\dd \nu) \dd r.
\end{equation}
    
\end{definition}

To state a solution of \eqref{eq:FKEn} in terms of the tuple $(\sfP,\sfJ^{\sfb\sfd},\sfJ^\sfh)$, we define the maps
\begin{align*}
    \sfT^{n,\pm}(\nu,x) &\coloneq \left(\nu\pm n^{-1} \delta_x,x\right),\qquad (\nu,x)\in\varGamma\times X,\\
    \sfT^{n,\sfh}(\nu,x,y) &\coloneq \left(\nu+n^{-1}(\delta_y-\delta_x),y,x\right),\qquad (\nu,x,y)\in\varGamma\times X\times X,
\end{align*}
describing the birth, death, and hopping transitions, and jump intensity measures
\begin{gather*}
    \vartheta^{n,\sfb}_{\sfP}(\dd \nu,\dd x)\coloneq\sfb_n(\nu,\dd x)\sfP(\dd \nu), \qquad \vartheta^{n,\sfd}_{\sfP}(\dd \nu,\dd x)\coloneq\sfd_n(\nu,x)\nu(\dd x)\sfP(\dd \nu),\\
    \vartheta^{n,\sfh}_{\sfP}(\dd \nu,\dd x,\dd y) \coloneq \sfh_n(\nu,x,\dd y)\nu(\dd x)\sfP(\dd \nu).
\end{gather*}
Note that $\sfT^{n,-}_{\#}\vartheta^{n,\sfd}_{\sfP}$ and $\sfT^{n,\sfh}_{\#} \vartheta^{n,h}_{\sfP}$ are also well-defined measures.

\begin{lemma}\label{lm:solf}
    A curve $\sfP$ is a solution to \eqref{eq:FKEn} if and only if $(\sfP,\sfJ)\in \mathsf{CE}_n$ with 
    \begin{equation}\label{eq:sol_flux}
        \sfJ^{\sfb\sfd}_t=\vartheta^{n,\sfb}_{\sfP_t}-\sfT^{n,-}_{\#}\vartheta^{n,\sfd}_{\sfP_t}, \qquad \sfJ^{\sfh}_t= \frac{1}{2}\left(\vartheta^{n,\sfh}_{\sfP_t}-\sfT^{n,\sfh}_{\#} \vartheta^{n,\sfh}_{\sfP_t}\right),
    \end{equation}
    for almost every $t\in[0,T]$.
\end{lemma}

\begin{proof}
    For any $\sfP$ satisfying the mass bound \eqref{eq:sol_b1} and any $F\in B_b(\varGamma)$, we have that
    \begin{align*}
        \iint_{\varGamma\times X} \dnabla_{n,\sfb\sfd}F(\nu,x)\dd \left(\vartheta^{n,b}_{\sfP}-\sfT^{n,-}_{\#}\vartheta^{n,d}_{\sfP}\right) &= \iint_{\varGamma\times X} n\bigl[F(\nu+n^{-1}\delta_x)-F(\nu)\bigr]\sfb_n(\nu,\dd x)\sfP(\dd \nu) \\
        &\qquad +\iint_{\varGamma\times X} n\bigl[F(\nu-n^{-1}\delta_x)-F(\nu)\bigr]\sfd_n(\nu,x)\nu(\dd x)\sfP(\dd \nu) \\
        &= \int_\varGamma L_{\sfb\sfd}^nF(\nu)\,\sfP(\dd\nu).
    \end{align*}
    Similarly, for the hopping part, we have that
     \begin{align*}
        -\iiint_{\varGamma\times X^2} \dnabla_{n,\sfh}F(\nu,x,y)\,\sfT^{n,h}_{\#}\vartheta^{n,h}_{\sfP}(\dd\nu\dd x\dd y) &=\iiint_{\varGamma\times X^2} \dnabla_{n,\sfh}F(\nu,x,y)\, \vartheta^{n,h}_{\sfP}(\dd\nu\dd x\dd y) \\ 
        &=\iiint_{\varGamma\times X^2} \dnabla_{n,\sfh}F(\nu,x,y)\, \sfh_n(\nu,x,\dd y)\nu(\dd x)\sfP(\dd \nu) \\
        &= \int_\varGamma L^n_\sfh F(\nu)\,\sfP(\dd\nu),
    \end{align*}
    thereby concluding the proof.
\end{proof}

\subsection{Detailed balance}\label{ss:rever}

We now state the detailed balance condition for both the birth-and-death and hopping parts separately, under Assumption \ref{ass:db}.

\begin{lemma}\label{lm:rever}
    Both jump kernels $\kappa^n_{\sfb\sfd}$ and $\kappa^n_{\sfh}$ satisfy the detailed balance condition w.r.t.\ $\Pi^n$. Consequently, the jump kernel $\kappa^n$ satisfies the detailed balance condition w.r.t.\ $\Pi^n$.

    In particular, we have that
    \[ 
        \sfT^{n,-}_{\#} \vartheta^{n,\sfd}_{\Pi^n}=\vartheta^{n,\sfb}_{\Pi^n}, \qquad \sfT^{n,\sfh}_{\#} \vartheta^{n,\sfh}_{\Pi^n} =\vartheta^{n,\sfh}_{\Pi^n}. 
    \]
\end{lemma}
\begin{proof}
    The symmetry of the birth-and-death part $\kappa_{\sfb\sfd}^n(\nu,\dd\eta)\,\Pi^n(\dd\nu)$ follows from \cite[Lemma 17.12]{ThesisHoeksema2023}. As for the hopping part, we show that
    \[
    	\int_{\varGamma} G(\nu,\eta)\,\kappa^n_{\sfh}(\nu,\dd \eta)\,\Pi^n(\dd \nu) = \int_{\varGamma} G(\eta,\nu)\,\kappa^n_{\sfh}(\nu,\dd \eta)\,\Pi^n(\dd \nu),
    \]
    which implies the symmetry of the measure $\kappa^n_{\sfh}(\nu,\dd \eta)\,\Pi^n(\dd \nu)$. To simplify notation, we set
    \begin{align*}
    		\widehat\sfh_n(\x,x,\dd y) \coloneq \sfh_n(\Lambda^n(\x),x,\dd y),\qquad
    		\widehat G_n(\x,\sigma) \coloneq G(\Lambda^n(\x),\Lambda^n(\sigma)),
    \end{align*}
    and $\x_j^y \coloneq (x_1,\ldots,x_{j-1},y,x_{j+1},\ldots,x_N)$ for any $\x\in\X^N$, $y\in X$. We then have that
    \begin{align*}
    \int_{\varGamma} G(\nu,\eta)\,\kappa^n_{\sfh}(\nu,\dd \eta)\,\Pi^n(\dd \nu) & \\
    &\hspace{-6em}= e^{-n|\pi|}\sum_{N=0}^\infty \frac{n^N}{N!}\int_{X^N} \sum_{j=1}^N \int_X \widehat G_n(\x,\x_j^y)\,\widehat\sfh_n(\x,x_j,\dd y)\,\pi^{\otimes N}(\dd x_1\cdots\dd x_N) \\
    &\hspace{-6em}\stackrel{(1)}{=} e^{-n|\pi|}\sum_{N=0}^\infty \frac{n^N}{N!}\int_{X^N} \sum_{j=1}^N \int_X \widehat G_n(\x,\x_j^{y})\,\widehat\sfh_n(\x_j^{y},y,\dd x_j)\,\pi(\dd y)\,\pi^{\otimes (N-1)}(\widehat{\dd x_j}) \\
    &\hspace{-6em}\stackrel{(2)}{=} e^{-n|\pi|}\sum_{N=0}^\infty \frac{n^N}{N!}\int_{X^N} \sum_{j=1}^N \int_X \widehat G_n(\sigma_j^{ y},\sigma)\,\widehat\sfh_n(\sigma,x_j,\dd y)\,\pi^{\otimes N}(\dd x_1\cdots\dd x_N) \\
    &\hspace{-6em}= \int_{\varGamma} G(\eta,\nu)\,\kappa^n_{\sfh}(\nu,\dd \eta)\,\Pi^n(\dd \nu)
\end{align*}
where $(1)$ follows from Assumption~\ref{ass:db} and $(2)$ follows from swapping the role of $x_j$ and $y$.
\end{proof}

\subsection{Generalized gradient flow formulation}\label{ss:nggf}

As mentioned in the introduction, a generalized gradient structure consists of a dissipation potential $\sfR_n$, a dual dissipation potential $\sfR_n^*$, and a driving functional $\sfE_n$. 

We begin by recalling the driving functional 
\[
  \calP(\varGamma)\ni\sfP\mapsto  \sfE_n(\sfP) = n^{-1}\Ent(\sfP,\Pi^n).
\]
To state the dissipation potentials, we will need to introduce various objects. First, we define  
\[ 
    \R\times[0,+\infty)\times[0,+\infty)\ni(w,u,v)\mapsto \Upsilon(w,u,v)\coloneq\begin{cases}
        0 & \text{if $w=0$}, \\
        \uppsi\left( w / \sqrt{u v}\right)\sqrt{u v} & \text{if $u,v>0$}, \\
        +\infty &\text{if $w\neq 0$ and $uv=0$},
    \end{cases}
    \]
where $\uppsi$ is the Legendre dual of 
\[
    \R\ni z\mapsto\uppsi^*(z)\coloneq2(\cosh(z/2)-1).
\]
It is straightforward to check that $\Upsilon$ is jointly convex, lower semicontinuous, and $1$-homogeneous (cf.\ \cite{PRST2022}).

\subsubsection*{Birth-and-death dissipation potentials} We define the birth-and-death dissipation and dual dissipation potentials as
\begin{equation}
    \begin{aligned}
        \calP(\varGamma_n)\times \calM(\varGamma\times X)\ni (\sfP,\sfJ)\mapsto \sfR_{n,\sfb\sfd}(\sfP,\sfJ) & \coloneq 2\iint_{\varGamma\times X} \Upsilon(\sfJ/2,\vartheta^{n,\sfb}_{\sfP},\sfT^{n,-}_{\#}\vartheta^{n,\sfd}_{\sfP}), \\
         \calP(\varGamma_n)\times B_b(\varGamma\times X)\ni (\sfP,\upzeta)\mapsto\sfR_{n,\sfb\sfd}^*(\sfP,\upzeta)& \coloneq 2\iint_{\varGamma\times X} \uppsi^*(\upzeta) \dd \sqrt{\vartheta^{n,\sfb}_{\sfP} \left(\sfT^{n,-}_{\#}\vartheta^{n,\sfd}_{\sfP}\right)},
    \end{aligned}
\end{equation}
where for any two measures $\mu,\nu$, the square root of its product is given by
$
    \sqrt{\mu\nu} = \sqrt{\tfrac{\dd\mu}{\dd\sigma}\tfrac{\dd\nu}{\dd\sigma}}\sigma
$
for any common dominating measure $\sigma$, i.e., $\mu,\nu\ll\sigma$ (see \cite[Section~2.3]{PRST2022} for a detailed discussion on concave transformations of measures). We further define the birth-and-death Fisher information as
\[
    \calP(\varGamma_n)\ni \sfP\mapsto \sfD_{n,\sfb\sfd}(\sfP) \coloneq 4 \calH^2(\vartheta^{n,\sfb}_{\sfP},\sfT^{n,-}_{\#}\vartheta^{n,\sfd}_{\sfP})\in[0,+\infty],
\]
where $\calH$ is the Hellinger distance between two measures defined as
$$
     \calH^2(\mu, \nu) = \frac{1}{2} \int \left( \sqrt{\frac{\dd\mu}{\dd\sigma}} - \sqrt{\frac{\dd\nu}{\dd\sigma}} \right)^2 \dd \sigma
$$
for some common dominating measure $\sigma$.

\subsubsection*{Hopping dissipation potentials}
In a similar fashion, we define the hopping dissipation and dual dissipation potentials as
\begin{equation}
    \begin{aligned}
        \calP(\varGamma_n)\times \calM(\varGamma\times X^2)\ni (\sfP,\sfJ)\mapsto\sfR_{n,\sfh}(\sfP,\sfJ)&\coloneq\iiint_{\varGamma\times X^2} \Upsilon(\sfJ,\vartheta^{n,\sfh}_{\sfP},\sfT^{n,\sfh}_{\#}\vartheta^{n,\sfh}_{\sfP}), \\
        \calP(\varGamma_n)\times B_b(\varGamma\times X^2)\ni(\sfP,\upzeta)\mapsto \sfR_{n,\sfh}^*(\sfP,\upzeta)& \coloneq\iiint_{\varGamma\times X^2} \uppsi^*(\upzeta) \dd \sqrt{\vartheta^{n,\sfh}_{\sfP} \left(\sfT^{n,\sfh}_{\#}\vartheta^{n,\sfh}_{\sfP}\right)},
    \end{aligned}
\end{equation}
and the hopping Fisher information as 
\[
    \calP(\varGamma_n)\ni \sfP\mapsto \sfD_{n,\sfh}(\sfP) \coloneq 2 \calH^2(\vartheta^{n,\sfh}_{\sfP},\sfT^{n,\sfh}_{\#}\vartheta^{n,\sfh}_{\sfP})\in[0,+\infty].
\]

\begin{remark}
If $\sfP = U\Pi^n$, then by Lemma \ref{lm:rever}, we have that
    \[
       \sqrt{\vartheta^{n,\sfb}_{\sfP} \left(\sfT^{n,-}_{\#}\vartheta^{n,\sfd}_{\sfP}\right)}=  \sqrt{U^{\sfb\sfd}(\nu)U(\nu) }\, \sfb_n(\nu,\dd x)\Pi^n(\dd \nu),
    \]
    allowing us to express the Fisher information as
    \begin{align*}
        \sfD_{n,\sfb\sfd}(\sfP) &= 2\iint_{\varGamma\times X} \left(\sqrt{U^{\sfb\sfd}(\nu,x)}-\sqrt{U(\nu)}\right)^2 \sfb_n(\nu,\dd x)\Pi^n(\dd \nu) \\
        &= 2n^{-1}\iint_{\varGamma\times X} \left(\dnabla_{n,\sfb\sfd}\sqrt{U}(\nu,x)\right)^2 \sfb_n(\nu,\dd x)\Pi^n(\dd \nu),
    \end{align*}
    where $U^{\sfb\sfd}(\nu,x)\coloneq U(\sfT^{n,+}(\nu,x))$.
    
As for the hopping case, if $\sfP = U\Pi^n$, then 
    \[
    \sqrt{\vartheta^{n,\sfh}_{\sfP} \left(\sfT^{n,\sfh}_{\#}\vartheta^{n,\sfh}_{\sfP}\right)} = \sqrt{U^\sfh(\nu)U(\nu)}\, \sfh_n(\nu,x,\dd y)\nu(\dd x)\Pi^n(\dd \nu),
    \]
    and the Fisher information reads as
    \begin{align*}
        \sfD_{n,\sfh}(\sfP) &= \iiint_{\varGamma\times X^2} \left(\sqrt{U^\sfh(\nu)}-\sqrt{U(\nu)}\right)^2 \,\sfh_n(\nu,x,\dd y)\nu(\dd x)\Pi^n(\dd \nu) \\
        &= n^{-1}\iiint_{\varGamma\times X^2} \left(\dnabla_{n,\sfh}\sqrt{U}(\nu,x)\right)^2 \,\sfh_n(\nu,x,\dd y)\nu(\dd x)\Pi^n(\dd \nu),
    \end{align*}
    where $U^\sfh(\nu,x)\coloneq U(\sfT^{n,\sfh}(\nu,x))$.
\end{remark}

Putting the birth-and-death and hopping part together, we obtain the formulation of our action-functional
\begin{equation*}
    \sfR_n(\sfP,\sfJ)\coloneq\sfR_{n,\sfb\sfd}(\sfP,\sfJ^{\sfb\sfd})+\sfR_{n,\sfh}(\sfP,\sfJ^{\sfh}),
\end{equation*}
and total Fisher information $\sfD_{n}\coloneq\sfD_{n,\sfb\sfd}+\sfD_{n,\sfh}$.

Formally, one can now check that \eqref{eq:FKEn} can be written via the kinetic relation as
\begin{align*}
\partial_t\sfP_t + \ddiv_{n,\sfb\sfd}\, \sfJ_t^{\sfb\sfd} + \ddiv_{n,\sfh}\, \sfJ_t^{\sfh} = 0,\qquad \left\{\;\;
\begin{aligned}
\sfJ_t^{\sfb\sfd} &=\partial_2\sfR_{n,\sfb\sfd}^*(\sfP_t,-\dnabla_{n,\sfb\sfd} \sfE_n'(\sfP_t)),\\ \sfJ_t^{\sfh}&=\partial_2\sfR_{n,\sfh}^*(\sfP_t,-\dnabla_{n,\sfh} \sfE_n'(\sfP_t)).
\end{aligned}\right.
\end{align*}
Indeed, using the shorthand notation $U,U^{\sfb\sfd}$ and $U^\sfh$, we find 
\begin{equation*}
    \sfE_n'(\sfP)=n^{-1}\log U, \qquad \dnabla_{n,\sfb\sfd} \sfE_n'(\sfP)=\log U^{\sfb\sfd}-\log U, \qquad \dnabla_{n,\sfh} \sfE_n'(\sfP)=\log U^{\sfh}-\log U,
\end{equation*}
and for the derivatives of the dual dissipation potentials,
\begin{equation*}
\begin{aligned}
      \partial_2 \sfR_{n,\sfb\sfd}^*(\sfP,-\dnabla_{n,\sfb\sfd} \sfE_n'(\sfP))&=\bigl(U(\nu)-U^{\sfb\sfd}(\nu,x)\bigr)\sfb(\nu,\dd x)\Pi^n(\dd \nu)=\vartheta^{n,\sfb}_{\sfP}-\sfT^{n,-}_{\#}\vartheta^{n,\sfd}_{\sfP}, \\
      \partial_2 \sfR_{n,\sfh}^*(\sfP,-\dnabla_{n,\sfh} \sfE_n'(\sfP))&=\frac{1}{2}\bigl(U(\nu)-U^{\sfh}(\nu,x)\bigr)\sfh(\nu,x,\dd y)\nu(\dd x)\Pi^n(\dd \nu)=\frac{1}{2}\left(\vartheta^{n,\sfh}_{\sfP}-\sfT^{n,\sfh}_{\#}\vartheta^{n,\sfh}_{\sfP}\right).
\end{aligned}
\end{equation*}
which are the correct fluxes from Lemma \ref{lm:solf}. By a similar calculation, one can check that 
\[
    \sfD_{n,\sfb\sfd}(\sfP)=\sfR_{n,\sfb\sfd}^*(\sfP,-\dnabla_{n,\sfb\sfd} \sfE_n'(\sfP)), \qquad \sfD_{n,\sfh}(\sfP)=\sfR_{n,\sfh}^*(\sfP,-\dnabla_{n,\sfh} \sfE_n'(\sfP)).
\]

The rigorous version of the above statements requires the aforementioned variational formulation. Namely, for any tuple $(\sfP,\sfJ^{\sfb\sfd},\sfJ^\sfh)\in \mathsf{CE}_n$, let $\mathscr{I}_n(\sfP,\sfJ)=\mathscr{I}_n([0,T];\sfP,\sfJ)$, with
\begin{equation}
    \mathscr{I}_n([s,t];\sfP,\sfJ)\coloneq \int_s^t \left(\sfR_n(\sfP_r,\sfJ_r) + \sfD_n(\sfP_r)\right)\dd r + \sfE_n(\sfP_t) - \sfE_n(\sfP_s) = 0.
\end{equation}

The functional corresponds to a well-posed gradient system if $\mathscr{I}_n([s,t],\cdot)\geq 0$ for all $s,t\in [0,T]$ and has a unique null-minimizer, which is called the corresponding \emph{generalized gradient flow solution}. For the latter, the energy-dissipation balance holds:
\[ 
\mathscr{I}_n([s,t];\sfP,\sfJ)=0\qquad\text{for every $[s,t]\subset[0,T]$}.
\]

We now have the following result. 

\begin{theorem}\label{th:fke_main}
Let $\sfP_0\in\text{dom}(\sfE_n)$. Then the functional $\mathscr{I}_n$ on $(\sfP,\sfJ)$ defines a well-posed generalized gradient system, and its unique generalized gradient flow solution is such that $\sfP$ is the solution to \eqref{eq:FKEn} with fluxes given for almost every $t\in(0,T)$ by
\[\sfJ^{\sfb\sfd}_t=\vartheta^{n,\sfb}_{\sfP_t}-\sfT^{n,-}_{\#}\vartheta^{n,\sfd}_{\sfP_t}, \qquad \sfJ^{\sfh}_t= \frac{1}{2}\left(\vartheta^{n,\sfh}_{\sfP_t}-\sfT^{n,\sfh}_{\#} \vartheta^{n,\sfh}_{\sfP_t}\right).\]
\end{theorem}

\begin{proof}
Recall from Remark \ref{rem:fke_edgeflux} that any $(\sfP,\sfJ)\in \mathsf{CE}_n$ can also be written in terms of edge fluxes $\sfJ^\circ$. Vice versa, any edge flux $\sfJ^\circ(\dd \nu\dd \eta)\ll \kappa^n(\nu,\dd \eta)\Pi(\dd \nu)$ can be decomposed into a birth/death $\sfJ^{\sfb\sfd}$ and $\sfJ^{\sfh}$ part, with 
\begin{align*}
        \sfJ^\circ_t(\dd \nu\dd \eta) &= \frac{n}{2} \int_{x\in X} \delta_{\nu+n^{-1}\delta_x}(\dd \eta)\, \sfJ_t^{\sfb\sfd}(\dd \nu,\dd x)-\frac{n}{2} \int_{x\in X} \delta_{\nu-n^{-1}\delta_x}(\dd \eta)\, \left(\sfT^{n,+}_{\#} \sfJ_t^{\sfb\sfd}\right)(\dd \nu,\dd x) \\
        &\qquad + n \iint_{(x,y)\in X^2} \delta_{\nu+ n^{-1}(\delta_y-\delta_x)}(\dd \eta)\, \sfJ_t^{\sfh}(\dd \nu,\dd x\dd y).
\end{align*}
Moreover, note that for any such $\sfP$,  
\begin{equation}\label{eq:fke_nkb}
    \sup_{t\in [0,T]} \int_{\varGamma} \|\nu\| \dd \sfP_t<\infty. 
\end{equation}

On the other hand, solutions to \eqref{eq:FKEn} in the sense of Definition \ref{def:sol-FKEn} are solutions to the forward Kolmogorov equation corresponding to a jump process with jump kernel $n \kappa^n$ satisfying detailed balance condition w.r.t.\ $\Pi^n$ and the bound \eqref{eq:fke_nkb}. The solution to \eqref{eq:FKEn} can be obtained by an approximation argument, where the jump kernel $\kappa^n$ is truncated so that the results of \cite{PRST2022} apply. Using the property that finite energy implies finite first moment (cf.\ \eqref{eq:moment}), one can pass to the limit in the truncation to obtain a curve satisfying \eqref{eq:sol_c}, a priori for any $C_c(\varGamma)$ functions, and a posteriori for any $B_b(\varGamma)$ functions due to the moment estimate.

Additionally, one can modify the arguments of \cite{PRST2022} to show that this solution is the unique generalized gradient flow solution for the corresponding well-posed functional 
\[     \mathscr{I}([s,t];\sfP,\sfJ^\circ)\coloneq \int_s^t \left(\sfR(\sfP_r,\sfJ^\circ_r) + \sfD(\sfP_r)\right)\dd r + \Ent(\sfP_t|\Pi^n) - \Ent(\sfP_s|\Pi^n),\]
with Fisher information $\sfD(\sfP) \coloneq 2 \calH^2(\sfP n \kappa^n, n \kappa^n \sfP)$, and dual dissipation potential
\begin{align*}
    \sfR^*(\sfP,\upzeta) \coloneq \iint_{\varGamma^2} \Psi^*(\upzeta) \dd \sqrt{(\sfP n \kappa^n)(n \kappa^n \sfP )},
\end{align*}
where $(\sfP \kappa^n)(\dd \nu\dd \eta)\coloneq\kappa^n(\nu,\dd \eta)\sfP(\dd \nu)$ and $(\kappa^n \sfP)(\dd \nu\dd \eta)\coloneq\kappa^n(\eta,\dd \nu)\sfP(\dd \eta)$. 

Note that for birth-and-death cases, we have
\begin{align*}
    \kappa^n_{\sfb\sfd}(\nu,\dd \eta)\sfP(\dd \nu) &=\int_{X} \delta_{\nu+n^{-1}\delta_x}(\dd \eta)\,\vartheta^{n,\sfb}_{\sfP}(\dd \nu, \dd x) + \int_{X} \delta_{\nu-n^{-1}\delta_x}(\dd \eta)\,\vartheta^{n,\sfd}_{\sfP}(\dd \nu, \dd x),\\
    \kappa^n_{\sfb\sfd}(\eta,\dd \nu)\sfP(\dd \eta) &=\int_{X} \delta_{\nu+n^{-1}\delta_x}(\dd \eta)\left(\sfT^{n,-}_{\#}\vartheta^{n,\sfd}_{\sfP}\right)(\dd \nu, \dd x) + \int_{X} \delta_{\nu-n^{-1}\delta_x}(\dd \eta)\left(\sfT^{n,+}_{\#}\vartheta^{n,\sfb}_{\sfP}\right)(\dd \nu, \dd x),
\end{align*} 
and for the hopping part
\begin{align*}
    \kappa^n_{\sfh}(\nu,\dd \eta)\sfP(\dd \nu) &=\iint_{(x,y)\in X^2} \delta_{\nu+n^{-1}(\delta_y-\delta_x)}(\dd \eta)\,\vartheta^{n,\sfh}_{\sfP_t}(\dd \nu,\dd x\dd y), \\
    \kappa^n_{\sfh}(\eta,\dd \nu)\sfP(\dd \eta) &=\iint_{(x,y)\in X^2} \delta_{\nu+n^{-1}(\delta_y-\delta_x)}(\dd \eta)\,(\sfT^{n,\sfh}_{\#}\,\vartheta^{n,\sfh}_{\sfP})(\dd \nu,\dd x\dd y).
\end{align*} 
Thus, we establish that in fact
\[\mathscr{I}([s,t];\sfP,\sfJ^\circ)=n \mathscr{I}_n([s,t];\sfP,\sfJ),\]
and therefore our solution in the sense of Definition \ref{def:sol-FKEn} exists and is the unique gradient flow solution corresponding to the functional $\mathscr{I}_n$.
\end{proof}

\subsection{Uniform estimates}

Here, we establish the necessary estimates to obtain compactness in Theorem \ref{th:convergence}. Let $\frakd$ be the metric on $\calP(\varGamma)$ defined by                       
\begin{equation}\label{eq:fke_dist}
    \frakd(\sfP,\mathsf{Q}) \coloneq\sup_{F\in {\mathcal{F}}} \int_{\varGamma} F \dd (\sfP-\mathsf{Q}),
\end{equation}  
with
\[\mathcal{F}\coloneq\Bigl\{ F\in \Cyl(\varGamma): \sup_{(\nu,x) \in \varGamma\times X} |\nabla_{\varGamma} F(\nu,x)| \leq  1 \Bigr\}.\]
\begin{remark}
    Using a similar superposition principle as in Section \ref{sec:superp}, one can show that $\frakd$ is in fact the $1$-Wasserstein metric on $\calP(\varGamma)$ induced by the total variation metric on $\varGamma$. 
\end{remark}
Note that $\frakd$ is narrowly lower semicontinuous, and convergence in $\frakd$ implies narrow convergence on narrowly compact subsets by density of $\calF$ in $C_b(\varGamma)$. We then have the following uniform bounds. 

\begin{lemma}\label{lm:fke_uniform}
    Let the family $(\sfP_0^n)_{n\geq 1}\subset\calP(\varGamma)$ be such that 
    \[
        \limsup_{n\to \infty} \sfE_n(\sfP_0^n) < +\infty. 
  \]
Then for the corresponding family $(\sfP^n,\sfJ^n)$ of generalized gradient flow solutions
\[
      \limsup_{n\to \infty} \sup_{t\in [0,T]} \sfE_n(\sfP_t^n) < \infty, \qquad
      \limsup_{n\to \infty } \frakd(\sfP_t^n,\sfP_s^n)\leq K |t-s|,
\]
for some constant $K>0$, independent of $s,t\in [0,T]$. 
\end{lemma}

\begin{proof}
The estimates for $\sfE_n(\sfP_t^n)$ follow directly from the fact that the driving energy $\sfE_n$ is decreasing along solutions. In particular, we obtain the bound
\[\limsup_{n\to \infty} \sup_{t\in [0,T]} \int_{\varGamma} \|\nu\|\, \sfP^n_t(\dd\nu) < +\infty.\]
Next, take any $F\in \calF$. Note that for any $\nu\in \varGamma_n$ and $x,y\in X$, 
\begin{align*}
| F(\nu +n^{-1}\delta_x)-F(\nu)|&\leq n^{-1}, \\
\int_{X}| F(\nu -n^{-1}\delta_x)-F(\nu)|\,\nu(\dd x)&\leq n^{-1} \|\nu\|,\\
\int_{X}| F(\nu +n^{-1}(\delta_y-\delta_x)-F(\nu)|\nu(\dd x)&\leq 2 n^{-1} \|\nu\|,
\end{align*}
which together with the uniform mass estimates and the uniform bounds on $\sfb_n$, $\sfd_n$ and $\sfh_n$, leads after substituting $F$ in \eqref{eq:sol_c} to a constant $K>0$ such that for every $s,t\in [0,T]$, $F\in \calF$, 
\[
   \left|\int_{\varGamma} F(\nu) \dd (\sfP^n_t-\sfP^n_s)\right| \leq K|t-s|.
\]
Taking the supremum over $F\in \calF$, we derive the asserted uniform estimate for the distance $\frakd$.     
\end{proof}

\section{Large-population limit}\label{sec:limit}
In this section, we consider the large-population limit for the forward Kolmogorov equation \eqref{eq:FKEn}. As indicated in the introduction, we expect the measure-valued jump processes to converge to 
\begin{equation}\label{eq:MF}\tag{\textsf{MFE}}
    \partial_t \nu_t = \scrV[\nu_t],\qquad \scrV = \kappa_{\sfb\sfd} + \kappa_\sfh\colon \varGamma\to\varGamma,
\end{equation}
with limit rates $\kappa_{\sfb\sfd}$ and $\kappa_\sfh$ given in \eqref{eq:rates-limit}, which is equivalent to \eqref{eq:mf-intro} if we take $\nu_t = u_t \pi$ and use the measure-dependent variants of the rates $\sfr_\sfb, \sfr_\sfd$ and  $\sfr_\sfh$ given in \eqref{eq:rates-density}. However, we do not work directly with the measure-valued jump processes. Instead, we consider the forward Kolmogorov equation \eqref{eq:FKEn} and the convergence to \eqref{eq:MF} is done in the sense that $\sfP^n_t \rightharpoonup \delta_{\nu_t}$ narrowly as $n\to \infty$. 

On the level of the measure-valued evolution, we expect \eqref{eq:FKEn} to converge to the Liouville equation
\begin{equation}\label{eq:Liouville}\tag{\textsf{LiE}}
    \partial_t \sfP_t + \mathrm{div}_{\varGamma} \big(\scrV \sfP_t \big) = 0 .
\end{equation}
It is important to recognize that \eqref{eq:MF} and \eqref{eq:Liouville} are closely related. If $\nu$ is a solution of \eqref{eq:MF}, then $\sfP_t = \delta_{\nu_t}$ solves \eqref{eq:Liouville}. Thus, \eqref{eq:Liouville} can be seen as the lifting of \eqref{eq:MF} in $\varGamma$ to $\calP(\varGamma)$.

\medskip

In Section~\ref{sec:liouville}, we show a formal derivation of \eqref{eq:Liouville} as a limit of \eqref{eq:FKEn} and present its generalized gradient flow formulation. Section~\ref{sec:mf} is dedicated to the well-possedness and variational formulation of \eqref{eq:MF}. We explore the connection between \eqref{eq:Liouville} and \eqref{eq:MF} in Section~\ref{sec:superp}, where we introduce the superposition principle and finalize the variational characterization of \eqref{eq:Liouville}. In Section~\ref{sec:convergence}, we prove the main convergence theorem, consequently resulting in Theorem~\ref{th:main}.

\subsection{Liouville equation and its generalized gradient flow formulation}\label{sec:liouville}
First, we give an intuitive explanation of why we expect \eqref{eq:Liouville} as the limit for \eqref{eq:FKEn}. Let $F:\varGamma\to\R$ be a cylindrical function of the form
\begin{equation}\label{eq:cyl-func}
    F(\nu) = g\big( \langle f_1, \nu\rangle, \dots, \langle f_M, \nu \rangle \big), \qquad g\in C_c^\infty(\R^M), \quad f_1 =1, ~~ f_m\in C_b(X), ~m= 2, \dots, M,
\end{equation}
and denote by $\Cyl(\varGamma)$ the collection of such functions.

We recall the birth-and-death generator $L_{\sfb\sfd}^n$ and see how it acts on a cylindrical function $F$. A simple application of the Taylor expansion provides 
\begin{align*}
    n\bigl[ F(\nu + n^{-1}\delta_x) - F(\nu) \bigr] &= \sum_{m=1}^M\partial_m g \big( \langle f_1,\nu\rangle,\ldots,\langle f_m,\nu\rangle \big) f_m(x) + O\big( \tfrac{1}{n} \big) = \nabla_\varGamma F(\nu,x) + O\big( \tfrac{1}{n} \big) \\
    n\bigl[ F(\nu - n^{-1}\delta_x) - F(\nu) \bigr] &= - \sum_{m=1}^M\partial_m g \big( \langle f_1,\nu\rangle,\ldots,\langle f_m,\nu\rangle \big) f_m(x) + O\Big( \tfrac{1}{n} \Big) = -\nabla_\varGamma F(\nu,x) + O\big( \tfrac{1}{n} \big),
\end{align*}
where we recall that $\nabla_{\varGamma} F$ is the $\varGamma$-gradient operator defined on cylindrical functions as
\begin{equation*}
    \nabla_{\varGamma} F(\nu,x) = \sum_{m=1}^M\partial_m g \big( \langle f_1,\nu\rangle,\ldots,\langle f_m,\nu\rangle \big) f_m(x),\qquad (\nu,x)\in\varGamma\times X.
\end{equation*}
Hence, we can write the action of $L^n_{\sfb\sfd}$ on cylindrical functions as
\begin{align*}
    L_{\sfb\sfd}^n F(\nu) 
    &= \int_X \nabla_{\varGamma} F(\nu, x) \, \big( \sfb_n(\nu,\dd x) - \sfd_n(\nu,x) \,\nu(\dd x) \big) + O \big( \tfrac{1}{n} \big).
\end{align*}
Formally, we expect to obtain the pointwise limit birth-and-death generator under appropriate assumptions on the rate functions,
\begin{align*}
    L_{\sfb\sfd}^\infty F(\nu)
    = \lim_{n\to\infty} L_{\sfb\sfd}^n F(\nu)
    &= \int_X \nabla_{\varGamma} F(\nu, x) \, \bigl(\sfb(\nu,\dd x) - \sfd(\nu,x) \nu(\dd x)\bigr) = \int_X \nabla_{\varGamma} F(\nu, x) \, \kappa_{\sfb\sfd} [\nu] (\dd x).
\end{align*}
For the hopping (Kawasaki) generator $L^n_\sfh$, similar calculations give
\begin{align*}
    L_\sfh^n F(\nu) 
    &= \iint_{X\times X} \big( \nabla_{\varGamma} F(\nu, y) - \nabla_{\varGamma} F(\nu, x) \big) \,\sfh_n(\nu,x,\dd y)\,\nu(\dd x) + O\bigl(\tfrac{1}{n}\bigr).
\end{align*}
We incorporate the notation $(\dnabla_\varGamma F)(\nu, x, y) \coloneq \nabla_{\varGamma} F(\nu, y) - \nabla_{\varGamma} F(\nu, x)$ and formally pass $n\to \infty$. We also employ the detailed balance assumption for $\sfh$ to see the following form of the limit hopping generator:
\begin{align*}
    L^\infty_{\sfh} F(\nu) &= \iint_{X^2} (\dnabla_{\varGamma}) F(\nu, x, y) \, \sfh(\nu, x, \dd y) \nu(\dd x) \\
    &= \int_X \nabla_\varGamma F(\nu, x) \Big[  \int_{y\in X} \sfh(\nu, y, \dd x) \nu(\dd y) - \int_{y\in X} \sfh(\nu, x, \dd y) \nu(\dd x)  \Big]
    = \int_X \nabla_{\varGamma} F(\nu, x) \, \kappa_\sfh[\nu] (\dd x).
\end{align*}
Together, we obtain limit generator $L^\infty = L_{\sfb\sfd}^\infty + L_{\sfh}^\infty$, which by duality, gives the Liouville equation \eqref{eq:Liouville}.

\medskip

Having formally obtained \eqref{eq:Liouville} as a reasonable limit for \eqref{eq:FKEn}, we defined the weak solutions for it. 

\begin{definition}[Weak solutions]\label{def:sol-LE} A narrowly continuous curve $P = (P_t)_{t\in [0, T]}$ is a weak solution to \eqref{eq:Liouville} if for all $0 \leq s < t \leq T$ and all $F\in \Cyl(\varGamma)$ it holds that
\begin{equation*}
    \int_{\varGamma} F \dd \sfP_t - \int_{\varGamma} F \dd \sfP_s = \int_s^t \iint_{\varGamma\times X} \nabla_{\varGamma} F (\nu, x)\, \scrV[\nu] (\dd x)\, \sfP_r(\dd \nu) \dd r.
\end{equation*}
\end{definition}

\begin{remark}The Liouville equation \eqref{eq:Liouville} is the transport equation associated with the measure-valued vector field $\scrV$. Therefore, the existence of weak solutions to the Liouville equation \eqref{eq:Liouville} can be proven using the well-posedness of the mean-field equation \eqref{eq:MF}. More precisely, if we define the flow map $\scrS: [0, T] \times \varGamma \to \varGamma$ satisfying the Lagrangian flow
\begin{equation}\tag{\textsf{LF}}
    \partial_t \scrS_t(\nu) = 
    \scrV(\scrS_t(\nu)),
\end{equation}
then one can show that $\sfP_t = (\scrS_t)_{\#} \sfP_0$ is a weak solution to \eqref{eq:Liouville} for any initial datum $\sfP_0\in\varGamma$.
\end{remark}

\begin{remark}\label{rem:li_mb}
The conditions $g\in C_c^\infty(\R^M)$ and $f_1 =1$ in \eqref{eq:cyl-func} imply that $\|\nu\|$ is bounded on the support of $F$ for any $F\in \mathrm{Cyl}(\varGamma)$, and in particular $\mathrm{Cyl}(\varGamma)\subset C_c(\varGamma)$. This frees us from having to prove convergence of mass moments when passing to the limit in \eqref{eq:FKEn}. 
\end{remark}

\subsubsection*{Generalized gradient flow formulation} \label{sec:LiE-GGF}
Now we discuss a variational formulation of the Liouville equation \eqref{eq:Liouville} given by its generalized gradient structure. This structure resembles the generalized gradient structure of the forward Kolmogorov equation \eqref{eq:FKEn} presented in Section~\ref{sec:particle-dynamics}. First we introduce the appropriate continuity equation 
\begin{equation}\label{eq:CE-limit}\tag{\textsf{CE$_\infty$}}
    \partial_t \sfP_t + \mathrm{div}_{\varGamma}\, \sfJ_t = 0,
\end{equation}
and consider arbitrary curves satisfying \eqref{eq:CE-limit} in the following sense.
\begin{definition}[Continuity equation] A tuple $(\sfP,\sfJ^{\sfb\sfd}, \sfJ^{\sfh})$ satisfies the continuity equation \eqref{eq:CE-limit}, or we say for short $(\sfP,\sfJ)\in \mathsf{CE}_\infty$, if
\begin{enumerate}
    \item The curve $[0, T] \ni t \mapsto \sfP_t \in \calP(\varGamma)$ is narrowly continuous with the mass bound:
     \[
        \sup_{t\in [0,T]} \int_{\varGamma} \|\nu\| \dd \sfP_t  < \infty.
    \]
    \item The Borel families $\{\sfJ^{\sfb\sfd}_t\}_{t\in [0,T]} \subset \calM(\varGamma\times X)$ and $\{\sfJ^{\sfh}_t\}_{t\in [0,T]} \subset \calM(\varGamma\times X^2)$ satisfy
    $$
        \int_0^T \|\sfJ^{\sfb\sfd}_t\| +\|\sfJ^{\sfh}_t\| \dd t < \infty.
    $$
    \item For any $0 \leq s < t \leq T$ and $F\in\Cyl (\varGamma)$, it holds that
    \begin{equation}\label{eq:CElim}
           \int_{\varGamma} F \dd P_t - \int_{\varGamma} F \dd P_s = \int_s^t \bigg[ \iint_{\varGamma\times X} \nabla_\varGamma F (\nu, x)\,\sfJ_r^{\sfb\sfd}(\dd\nu\dd x) +  \iiint_{\varGamma\times X^2} \dnabla_\varGamma F (\nu, x, y) \,\sfJ_r^{\sfh}(\dd\nu\dd x\dd y) \bigg]\dd r.
    \end{equation}
     
\end{enumerate}
\end{definition}

To define the energy-dissipation functional \eqref{eq:EDBlimit}, we specify the gradient structure $(\sfR_\infty, \sfR_\infty^*, \sfE_\infty)$ with driving energy
\begin{equation*}
    \calP(\varGamma)\ni\sfP\mapsto\sfE_\infty(\sfP) \coloneq \int_\varGamma \Ent(\nu | \pi) \,\sfP(\dd\nu).
\end{equation*}
Notice that if $\sfE_\infty(\sfP)<+\infty$, then $\Ent(\nu|\pi)<+\infty$ for $\sfP$-almost every $\nu\in\varGamma$, i.e.\ $\sfP$ is supported on finite measures having finite relative entropy w.r.t.\ the activity measure $\pi$.

In analogy to Section~\ref{sec:fke-sol}, we denote the jump intensity measures
\begin{gather*}
    \vartheta^{\sfb}_{\sfP}(\dd \nu,\dd x)\coloneq\sfb(\nu,\dd x)\sfP(\dd \nu), \qquad \vartheta^{\sfd}_{\sfP}(\dd \nu,\dd x)\coloneq\sfd(\nu,x)\nu(\dd x)\sfP(\dd \nu),\\
    \vartheta^{\sfh}_{\sfP}(\dd \nu,\dd x,\dd y) \coloneq \sfh(\nu,x,\dd y)\nu(\dd x)\sfP(\dd \nu),
\end{gather*}
with the limit rates $\sfb, \sfd, \sfh$ defined in Assumption~\ref{ass:db}. As before, the dissipation potential has two components: the birth-and-death dissipation potential $\sfR_{\sfb\sfd}$ and the hopping dissipation potential $\sfR_{\sfh}$, given by
\begin{equation}\label{eq:R-limit}
    \begin{aligned}
    \calP(\varGamma) \times \calM(\varGamma\times X) \ni (\sfP, \sfJ) \mapsto \sfR_{\infty,\sfb\sfd}(\sfP, \sfJ) &\coloneq 2 \iint_{\varGamma\times X} \Upsilon (\sfJ/2, \vartheta^\sfb_\sfP, \vartheta^\sfd_\sfP), \\ 
    \calP(\varGamma) \times \calM(\varGamma\times X^2) \ni (\sfP, \sfJ) \mapsto \sfR_{\infty,\sfh}(\sfP, \sfJ) &\coloneq \iiint_{\varGamma\times X^2} \Upsilon (\sfJ, \vartheta^\sfh_\sfP, \sfT_{\#}\vartheta^\sfh_\sfP),
    \end{aligned}
\end{equation}
where $\sfT(\nu,x,y)\coloneq (\nu,y,x)$. The total dissipation potential is then 
$$
    \sfR_\infty(\sfP, \sfJ) \coloneq \sfR_{\infty,\sfb\sfd}(\sfP, \sfJ^{\sfb\sfd}) + \sfR_{\infty,\sfh}(\sfP, \sfJ^\sfh).
$$
The corresponding dual dissipation potentials are
\begin{equation*}
    \begin{aligned}
    \calP(\varGamma) \times B_b(\varGamma\times X) \ni (\sfP, \upzeta) \mapsto \sfR_{\infty,\sfb\sfd}^* (\sfP, \upzeta) &\coloneq 2 \iint_{\varGamma\times X} \Psi^*(\upzeta) \dd \sqrt{\vartheta^\sfb_\sfP\,\vartheta^\sfd_\sfP}, \\
    \calP(\varGamma) \times B_b(\varGamma\times X^2) \ni (\sfP, \upzeta) \mapsto \sfR_{\infty,\sfh}^* (\sfP, \upzeta) &\coloneq \iiint_{\varGamma\times X^2} \Psi^*(\upzeta) \dd \sqrt{\vartheta^\sfh_\sfP \,\sfT_{\#}\vartheta^\sfh_\sfP}.
    \end{aligned}
\end{equation*}
Furthermore, we define the birth-and-death and hopping Fisher information by means of the Hellinger distance $\calH$:
\begin{equation}\label{eq:Fisher-limit}
    \begin{aligned}
        \calP(\varGamma) \ni \sfP \mapsto \sfD_{\infty,\sfb\sfd} (\sfP) &\coloneq 4 \calH^2(\vartheta^\sfb_\sfP, \vartheta^\sfd_\sfP) \in [0, +\infty], \\
        \calP(\varGamma) \ni \sfP \mapsto \sfD_{\infty,\sfh} (\sfP) &\coloneq 2 \calH^2(\vartheta^\sfh_\sfP, \sfT_{\#}\vartheta^\sfh_\sfP) \in [0, +\infty].
    \end{aligned}
\end{equation}
The total Fisher information is given by $\sfD(\sfP) = \sfD_{\sfb\sfd} (\sfP) + \sfD_{\sfh} (\sfP) $. 
Finally, we introduce the energy-dissipation functional for any $(\sfP, \sfJ) \in \mathsf{CE}_\infty$:
\begin{equation}\label{eq:EDF-limit}
    \mathscr{I}_\infty ([s,t];\sfP,\sfJ) \coloneq \int_s^t \left(\sfR_\infty(\sfP_r,\sfJ_r) + \calD_\infty(\sfP_r)\right)\dd r + \sfE_\infty(\sfP_t) - \sfE_\infty(\sfP_s).
\end{equation}

The components of the generalized gradient structure involved in $\mathscr{I}_\infty ([s,t];\sfP,\sfJ)$ were introduced in the lower-semicontinuous form, hinting that they can be obtained as Gamma-limits of the corresponding structure for \eqref{eq:FKEn} defined in Section~\ref{ss:nggf}. To make a clear connection between the structure $(\sfR_\infty, \sfR_\infty^*, \sfE_\infty)$ and the Liouville equation \eqref{eq:Liouville}, we first need to study the mean-field equation \eqref{eq:MF} in more detail. After establishing the variational formulation for \eqref{eq:MF} in Section~\ref{sec:mf}, we revisit the variational formulation of \eqref{eq:Liouville} and clarify the connection between the two and why $(\sfR_\infty, \sfR_\infty^*, \sfE_\infty)$ gives the appropriate gradient flow formulation for \eqref{eq:LiE}.

\subsection{Mean-field equation}\label{sec:mf}
In this section, we take a closer look at the mean-field equation with velocity field $\scrV=\kappa_{\sfb\sfd} + \kappa_\sfh$ with the limit rates derived above:
\begin{align*}
    \kappa_{\sfb\sfd}[\nu](\dd x) &= \sfb(\nu,\dd x) - \sfd(\nu,x)\nu(\dd x),\\
    \kappa_\sfh[\nu](\dd x) &= \int_{y\in X} \sfh (\nu, y,\dd x)\,\nu(\dd y) - \nu(\dd x) \int_{y\in X} \sfh (\nu,x,\dd y).
\end{align*}
Under Assumption~\ref{ass:db}, the velocity field $\scrV$ can easily be shown to be locally Lipschitz having linear growth, i.e.\ there exists constants $\ell_\scrV$ and $c_\scrV>0$ such that for any $\nu,\eta\in\varGamma$,
\begin{align*}
    \|\scrV[\nu]-\scrV[\eta]\| &\le \ell_\scrV\bigl(1 + \|\nu\| + \|\eta\|\bigl)\|\nu-\eta\|, \\
    \|\scrV[\nu]\| &\le c_\scrV\|\nu\|.
\end{align*}
Therefore, standard ODE theory in Banach spaces (cf.\ \cite{martin1976}) applies and we obtain the following result.

\begin{proposition}
    Let $\overline{u}\in L^1(X,\pi)$. Then there exist a unique differentiable curve $u\in C_b^1((0,T);L^1(X,\pi))$ satisfying the mean-field equation \eqref{eq:mf-intro} with $u_0 = \overline{u}$. 
    
    Consequently, the curve $t\mapsto \nu_t\coloneq u_t\pi$ is a solution of the mean-field equation \eqref{eq:MF} with $\nu_0=\overline{u}\pi$. 
    We call this curve a strong solution of \eqref{eq:MF}.
\end{proposition}

Under the detailed balance assumption \eqref{eq:DB-limit}, \eqref{eq:MF} has a gradient flow formulation. To make it precise, we consider the curves satisfying the continuity equation
\begin{equation}\label{eq:CE-MF}\tag{$\scrC\scrE$}
    \partial_t \nu_t = \lambda^{\sfb\sfd}_t + \overline{\mathrm{div}}\lambda^{\sfh}_t,
\end{equation}
where the discrete divergence
\begin{align*}
    \calM(X^2) \ni j &\mapsto (\overline{\mathrm{div}} j) (\dd x) \coloneq \int_{y\in X} j(\dd x, \dd y) - \int_{y\in X} j(\dd y, \dd x) \in \calM(X),
\end{align*}
is associated to the discrete gradient $\dnabla f(x,y)=f(y)-f(x)$ by duality.

\begin{definition}
    A triple $(\nu, \lambda^{\sfb\sfd}, \lambda^\sfh)$ satisfies \eqref{eq:CE-MF}, or we say $(\nu, 
    \lambda) \in \scrC\scrE$, if
    \begin{enumerate}
        \item The curve $[0, T] \ni t \mapsto \nu_t \in \varGamma$ is $\|\cdot\|$-absolutely continuous.
        \item The Borel families $\{ \lambda^{\sfb\sfd}_t\}_{t\in [0,T]}\subset \calM(X), ~\{ \lambda^{\sfh}_t\}_{t\in [0,T]} \subset \calM(X\times X)$ satisfy
        $$
            \int_0^T \|\lambda^{\sfb\sfd}_t\| + \|\lambda^{\sfh}_t\| \dd t < \infty.
        $$
        \item For any $0\leq s < t \leq T$ and $f\in C_b(X)$, it holds that
        $$
            \int_X f \dd \nu_t - \int_X f \dd \nu_s = \int_s^t \Big[ \int_X f \dd \lambda_r^{\sfb\sfd} + \iint_{X^2} \dnabla f \dd \lambda_r^{\sfh} \Big] \dd r.
        $$
    \end{enumerate}
\end{definition}

\medskip

We claim that \eqref{eq:MF} has a generalized gradient flow formulation $(\scrR, \scrR^*, \scrE)$ with driving energy $\scrE(\nu) = \Ent(\nu|\pi)$. The dissipation potential $\scrR$ comprises the birth-and-death and the hopping parts defined as
\begin{equation*}
    \begin{aligned}
        \varGamma\times \calM(X) \ni (\nu, \lambda^{\sfb\sfd}) \mapsto \scrR_{\sfb\sfd}(\nu, \lambda^{\sfb\sfd}) &\coloneq 2 \int_X \Upsilon(\lambda^{\sfb\sfd}/2, \sfb_\nu, \sfd_\nu), \\
        \varGamma\times \calM(X^2) \ni (\nu, \lambda^{\sfh}) \mapsto \scrR_{\sfh}(\nu, \lambda^{\sfh}) &\coloneq \int_X \Upsilon(\lambda^{\sfh}, \sfh_\nu, \sfT_{\#}\sfh_\nu),
    \end{aligned}
\end{equation*}
where we use the shorthand $\sfb_\nu(\dd x) \coloneq \sfb(\nu, \dd x)$, $\sfd_\nu(\dd x) \coloneq \sfd(\nu, x) \nu(\dd x)$ and $\sfh_\nu(\dd x, \dd y) \coloneq \sfh(\nu, x, \dd y) \nu(\dd x)$. Recall that $\sfT$ is the swapping map $\sfT(\nu,x,y)=(\nu,y,x)$.

The dual dissipation potential $\scrR^*$ is defined as
\begin{equation*}
    \begin{aligned}
    \varGamma\times B_b(X) \ni (\nu, w) \mapsto \scrR^*_{\sfb\sfd}(\nu, w) &\coloneq 2 \int_X \uppsi^*(w) \dd \sqrt{\sfb_\nu \sfd_\nu}, \\
    \varGamma\times B_b(X^2) \ni (\nu, w) \mapsto \scrR^*_{\sfh}(\nu, w) &\coloneq \iint_{X^2} \uppsi^*(w) \dd \sqrt{\sfh_\nu \sfT_{\#} \sfh_\nu},
    \end{aligned}
\end{equation*}
and the Fisher information is given by
\[
    \varGamma \ni \nu \mapsto \scrD(\nu) = 4 \calH^2(\sfb_\nu, \sfd_\nu) + 2\calH^2(\sfh_\nu, \sfT_{\#} \sfh_\nu ) \in [0, +\infty].
\]

Consequently, the energy-dissipation functional $\scrI$ reads
\begin{equation}\label{eq:EDF-MF}
    \scrI([s, t]; \nu, \lambda) \coloneq \int_s^t \scrR(\nu_r, \lambda_r) + \scrD(\nu_r) \dd r + \scrE(\nu_T) - \scrE(\nu_0).
\end{equation}

\begin{remark}
    If $\nu = u \pi$, then we have a simplified expression for the Fisher information:
    \begin{align*}
         \scrD(\nu) = 2 \int_X \left(\sqrt{u(x)} - 1\right)^2\sfr_\sfb (u, x) \pi (\dd x) + \iint_{X^2} \left( \sqrt{u(y)} - \sqrt{u(x)} \right)^2 \sfr_\sfh (u, x, y) \pi(\dd x) \pi(\dd y)
    \end{align*}
    with the density-dependent rates $\sfr_\sfb$ and $\sfr_\sfh$ defined in \eqref{eq:rates-density}.   Moreover, a formal calculation shows that 
    \[
    	\partial_2 \scrR^*\left(-\calE',-\dnabla \calE'\right)=\left(\sfb_{\nu}-\sfd_{\nu},\frac{1}{2}(\sfh_{\nu}-\sfT_{\#}\sfh_{\nu})\right), \qquad \scrD(\nu)=\scrR^*\left(-\calE',-\dnabla \calE'\right).
    \]    
\end{remark}

\begin{theorem}[Variational characterization of \eqref{eq:MF}] For any $(\nu, \lambda^{\sfb\sfd}, \lambda^\sfh) \in \scrC\scrE$ with $\scrE(\nu_0) < \infty$, we have $\scrI([s, t]; \nu, \lambda) \geq 0$ and $\nu_t$ is the unique strong solution to \eqref{eq:MF} if and only if $\scrI([0, T]; \nu, \lambda) = 0$ and 
\[ \lambda^{\sfb \sfd}=\sfb_{\nu}-\sfd_{\nu}, \qquad \lambda^{\sfh}=\frac{1}{2}(\sfh_{\nu}-\sfT_{\#}\sfh_{\nu}).\]

Moreover, if $\scrI([0, T]; \nu, \lambda) < \infty$, then the chain rule for $\scrE$ holds, i.e.\
$$
    \text{$t\mapsto \scrE(\nu_t)$ is absolutely continuous},\qquad \frac{\dd}{\dd t} \scrE(\nu_t) = \int_X \log \frac{\dd \nu_t}{\dd\pi} \dd \lambda_t\qquad\text{for almost every $t\in(0,T)$}.
$$
\end{theorem}

\begin{proof}
The proof follows by using a regularized entropy functional and combining arguments from \cite{PRST2022} for the jump setting, and \cite{hoeksema2023} for the case of birth and death. 
\end{proof}

\subsection{Superposition principle}\label{sec:superp}

The superposition principle, which makes the connection between the mean-field equation \eqref{eq:MF} and the Liouville equation \eqref{eq:Liouville}, is an essential tool for proving uniqueness and variational characterization of solutions. Intuitively, it says that a curve satisfying \eqref{eq:CE-limit} with finite action can be represented as a superposition of curves satisfying \eqref{eq:CE-MF}. 

With the variational formulation for \eqref{eq:MF} at hand, we realize that the energy-dissipation functional $\scrI_\infty$ given in \eqref{eq:EDF-limit} for \eqref{eq:Liouville} is "lifted" from the corresponding functional \eqref{eq:EDF-MF}. First observe that if $\sfE_\infty(\sfP) < \infty$, then the Fisher information has the following form:
\begin{align*}
    \int_\varGamma \scrD(\nu)\, \sfP(\dd\nu)
    &= 4 \int_\varGamma \calH^2(\sfb_\nu, \sfd_\nu) \,\sfP(\dd\nu) + 2 \int_\varGamma \calH^2(\sfh_\nu, \sfT_{\#} \sfh_\nu)  \,\sfP(\dd \nu) \\
    &= 4 \calH^2(\vartheta^\sfb_\sfP, \vartheta^\sfd_\sfP) + 2 \calH^2(\vartheta^\sfh_\sfP, \sfT_{\#}\vartheta^\sfh_\sfP) = \sfD_\infty (\sfP).
\end{align*}
It also holds that $\sqrt{\vartheta^\sfb_\sfP\,\vartheta^\sfd_\sfP}(\dd \nu \dd x) = \theta^{\sfb\sfd}_\nu (\dd x) \sfP(\dd \nu)$ and $\sqrt{\vartheta^\sfh_\sfP \, \sfT_{\#}\vartheta^\sfh_\sfP}
(\dd\nu\dd x\dd y) = \theta^\sfh_{\nu} (\dd x \dd y) \sfP(\dd\nu)$. If $\sfE_\infty(\sfP) < \infty$, one can also show the disintegration property for fluxes
\begin{align*}
    \sfJ^{\sfb\sfd}(\dd\nu \dd x) = \lambda^{\sfb\sfd}_\nu (\dd x) \sfP(\dd\nu) \quad \text{and} \quad \sfJ^{\sfh}(\dd\nu \dd x \dd y) = \lambda^{\sfh}_\nu (\dd x \dd y) \sfP(\dd\nu).
\end{align*}

Approaches based on superposition principles similar to ours were also employed in \cite{erbar2023gradient, erbaretal2016, hoeksema2023}, where they are applied to transport equations lifted from the Boltzmann equation,
mean-field jump dynamics, and mean-field birth-and-death dynamics respectively. In all these cases, the proof of the superposition principle relies on the superposition principle for
solutions of the continuity equation on $\R^\N$ established in \cite{ambrosio2014well}. 

\begin{theorem}[Superposition principle] \label{th:super-principle}
    Let $(\sfP, \sfJ) \in \textsc{CE}_\infty$ with 
    $$
        \int_0^T \sfR_\infty(\sfP_t, \sfJ_t)\,\dd t < \infty. 
    $$
Then there exists a Borel probability (path) measure $\mathbb{Q}\in\calP(C([0,T]; \varGamma))$ such that
\begin{enumerate}
    \item the time marginals are
    \[
        (e_t)_{\#} \mathbb{Q} = \sfP_t\qquad\forall\,t\in[0,T],
    \]
    where $e_t\colon C([0,T]; \varGamma)\to \varGamma$, $e_t(\nu)=\nu_t$ is the time evaluation map;
    \item $\mathbb{Q}$ is concentrated on $\|\cdot\|$-absolutely continuous curves $\nu\in\mathrm{AC} ([0,T]; (\varGamma, \|\cdot\|))$ with $(\nu, \lambda)\in \scrC\scrE$ for some Borel family $\{\lambda_t\}_{t\in[0,T]}\subset \calM(X)$ of fluxes;
    \item the following representation holds
    \[
        \int_0^T \sfR_\infty (\sfP_t, \sfJ_t) \dd t = \int_{C([0,T]; \varGamma)} \Big( \int_0^T \scrR (\nu_t, \lambda_t) \dd t \Big) \, \mathbb{Q}(\dd \nu).
    \]
\end{enumerate}
\end{theorem}
\begin{proof}
   Let $\{f_m\}_{m\in\N} \subset C_b(\Omega)$ be a countable dense set with with $f_1 = 1$. We define
    \begin{equation*}
        \varGamma\ni \nu \mapsto \mathbb{T}(\nu) \coloneq \Big( \int_X f_1 \dd\nu, \int_x f_2 \dd\nu, \dots \Big) \in \R^{\N}.
    \end{equation*}
    We further set $\sigma_t \coloneq \mathbb{T}_{\#} \sfP_t \in \calP(\R^{\N})$ and define a vector field $W\colon [0,T]\times \R^n \to \R^n$ via its components, $W_m\coloneq W_m^{\sfb\sfd} + W_m^\sfh$, $m\in\N$, where
    \[
        W^{\sfb\sfd}_m(t, z) \coloneq \int_X f_m(x) \lambda^{\sfb\sfd}_{t,\mathbb{T}^{-1}(z)} (\dd x),\qquad
        W^{\sfh}_m(t, z) \coloneq \iint_{X^2} \bigl(f_m(y) - f_m(x)\bigr) \lambda^{\sfh}_{t,\mathbb{T}^{-1}(z)} (\dd x \dd y).
    \]
    The pair $(\sigma, W)$ satisfies the continuity equation in the sense that, for all $g\in\Cyl (\R^{\N})$, we have
    $$
        \int_{\R^\N} g \dd\sigma_t - \int_{\R^\N} g \dd\sigma_s  = \int_s^t \int_{\R^\N} \nabla g \cdot W_r \,\sigma_r (\dd\nu) \dd r
    $$
    To see this we define $F \coloneq g\circ \mathbb{T}$ and use that $(\sfP,\sfJ)\in\text{CE}_\infty$ to deduce
    \begin{align*}
        \int_{\R^\N} g \dd\sigma_t - \int_{\R^\N} g \dd\sigma_s
        &= \int_\varGamma F \dd\sfP_t - \int_\varGamma F \dd\sfP_s \\
        &= \int_s^t \bigg[ \iint_{\varGamma\times X} \nabla_\varGamma F (\nu, x)\,\sfJ_r^{\sfb\sfd}(\dd\nu\dd x) +  \iiint_{\varGamma\times X^2} \dnabla_\varGamma F (\nu, x, y) \,\sfJ_r^{\sfh}(\dd\nu\dd x\dd y) \bigg]\dd r \\
        &= \int_s^t \bigg[ \int_{\varGamma} \sum_{m=1}^M \partial_m g (\mathbb{T}(\nu)) \int_X f_m(x) \lambda^{\sfb\sfd}_r(\dd x) \,\sfP_r (\dd\nu) \\
        &\hspace{2cm}+  \int_{\varGamma} \sum_{m=1}^M \partial_m g (\mathbb{T}(\nu)) \iint_{X^2} \bigl(f_m(y) - f_m(x)\bigr) \lambda_r^{\sfh}(\dd x\dd y) \,\sfP_r (\dd\nu) \bigg]\dd r \\
        &= \int_s^t \bigg[ \int_{\R^\N} \nabla g(z) \cdot W^{\sfb\sfd}_r(z) \,\sigma_r (\dd\nu) +  \int_{\R^\N} \nabla g(z) \cdot W^{\sfh}_r(z) \,\sigma_r (\dd\nu) \bigg]\dd r.
    \end{align*}
    Moreover, the defined vector field $W$ belongs to $L^1([0,T]\times\R^\N, \calL|_{[0,T]}\otimes\sigma_t)$, since
    \begin{align*}
        \int_0^T \int_{\R^\N} |W_{m}(t, z)| \dd \sigma_t \dd t
        = \int_0^T \int_\varGamma |W_{m}(t, \mathbb{T}(\nu))|\, \sfP_t(\dd\nu) \dd t
        \leq \int_0^T \sfR_\infty(\sfP_t, \sfJ_t)\,\dd t) < \infty. 
    \end{align*}

    We are now in the position to apply \cite[Theorem~7.1]{ambrosio2014well} to conclude that there exists a Borel probability measure $\mathbb{q}$ in $\calP(C([0,T]; \R^\N))$ satisfying $(e_t)_{\#} \mathbb{q} = \nu_t$ for all $t\in [0,T]$, and is concentrated on abosolutely continuous curves $z\in AC([0,T];\R^\N)$, which are solutions to $\Dot{z}_t = W_t(z_t)$ for almost every $t\in [0,T]$.
\end{proof}

\begin{theorem}[Variational characterization of \eqref{eq:Liouville}]\label{th:superposition} For any $(\sfP, \sfJ)\in \mathsf{CE}_\infty$ with $\sfP_0\in \text{dom}(\sfE_\infty)$, the energy-dissipation functional $\scrI_\infty$ is finite if and only if there exists a Borel probability measure $\mathbb{Q} \in \calP(C([0,T];\varGamma))$ such that
\begin{enumerate}
    \item the time marginals are $(e_t)_{\#} \mathbb{Q} = \sfP_t$ for all $t\in [0, T]$;
    \item $\mathbb{Q}$ is concentrated on the family of $\|\cdot\|$-absolutely continuous curves $\nu \in \mathrm{AC}([0, T]; (\varGamma, \|\cdot\|))$ with $(\nu, \lambda) \in \scrC\scrE$, where $\lambda=\lambda^{\sfb\sfd}+\lambda^\sfh$ satisfies $\sfJ^{\sfb\sfd}_t = \lambda^{\sfb\sfd}_t \sfP_t $ and $\sfJ_t^\sfh = \lambda_t^\sfh\sfP_t$ for almost every $t\in[0,T]$;
    \item the following representation holds
    $$
        \scrI_\infty([s,t]; \sfP, \sfJ) = \int_{C([0,T];\varGamma)} \scrI ([s,t]; \nu, \lambda) \dd \mathbb{Q}.
    $$
\end{enumerate}
In particular, $\scrI_\infty ([0,T]; \sfP, \sfJ) \geq 0$ and equal to $0$ if and only if $\sfP_t$ is the weak solution to \eqref{eq:Liouville} with 
\begin{align*}
    \sfP_t &= (\scrS_t)_{\#} \sfP_0\qquad\text{for every $t\in[0,T]$},\;\;\text{and fluxes}\\ 
     \sfJ^{\sfb\sfd}_{\sfP_t} &=\vartheta^{\sfb}_{\sfP_t}-\vartheta^{\sfd}_{\sfP_t}, \qquad \sfJ^{\sfh}= \frac{1}{2}\left(\vartheta^{\sfh}_{\sfP_t}-\sfT_{\#}\vartheta^{\sfh}_{\sfP_t}\right) \quad \text{for almost every } t\in [0,T].
\end{align*}
\end{theorem}

\subsection{Convergence results}\label{sec:convergence}

The generalized gradient structures introduced in Section~\ref{ss:nggf} and Section~\ref{sec:liouville} are particularly useful for rigorously proving the convergence of \eqref{eq:FKEn} to \eqref{eq:Liouville}. We employ a form of convergence called evolutionary Gamma-convergence. It was first introduced in \cite{sandier2004gamma} and led to numerous subsequent studies surveyed in \cite{mielke2016evolutionary, serfaty2011gamma}. Let us mention that other notions of convergence were introduced in this context known as EDP convergence \cite{dondl2019gradient,mielke2021exploring}, which requires further limsup estimates and is not considered in this paper. We follow the strategy of the evolutionary Gamma-convergence: establishing the liminf estimate for the energy-dissipation functional and deducing the convergence of gradient flow solutions.

\begin{theorem}\label{th:convergence} Let $\{(\sfP^n, \sfJ^n)\}_{n\in\N}$ be a sequence of generalized gradient flow solutions of \eqref{eq:FKEn} 
with initial data $\{\sfP^n_\text{in}\}_{n\in\N}$ and suppose there exists $\sfP_\text{in}\in \text{dom} (\sfE_\infty)$ such that
$$
    \sfP^n_\text{in} \rightharpoonup \sfP_\text{in},\qquad \lim_{n\to\infty} \sfE_n (\sfP^n_\text{in}) = \sfE_\infty(\sfP_\text{in}).
$$
Then there exist a (not relabelled) subsequence$\{(\sfP^n, \sfJ^n)\}_{n\in\N}$ and a limit pair $(\sfP, \sfJ)$ such that
\begin{enumerate}
    \item $(\sfP, \sfJ) \in \mathsf{CE}_\infty$, where
    \begin{itemize}
        \item $\sfP^n_t \rightharpoonup \sfP_t$ narrowly for all $t\in [0,T]$;
        \item $\sfJ^{n,\sfb\sfd}_t(\dd\nu\dd x) \rightharpoonup^* \sfJ^{\sfb\sfd}_t (\dd\nu\dd x)$ weakly-* in $\calM(\varGamma\times X)$ for all $t\in [0,T]$;
        \item $\sfJ^{n,\sfh}_t (\dd\nu\dd x\dd y)\rightharpoonup^* \sfJ^{\sfh}_t (\dd\nu\dd x\dd y)$ weakly-* in $\calM(\varGamma\times X^2)$ for all $t\in [0,T]$.
    \end{itemize}
    \item The following $\liminf$ inequality holds
    \[
        \mathscr{I}_\infty ([s,t];\sfP,\sfJ)\le \liminf_{n\to\infty} \mathscr{I}_n([s,t];\sfP^n,\sfJ^n)\qquad\text{for every $[s,t]\subset[0,T]$}.
    \]
    \item $(\sfP, \sfJ)$ is the generalized gradient flow solution of the Liouville equation \eqref{eq:Liouville} with initial datum $\sfP_{in}$.
\end{enumerate}
\end{theorem}
\begin{proof}
We first prove the compactness statement of (1). Consider a sequence $\{(\sfP^n, \sfJ^n)\}_{n\in\N}$ of  generalized gradient flow solutions of \eqref{eq:FKEn} with initial data $\{\sfP^n_\text{in}\}_{n\in\N}$ such that $\sup_{n\in\N} \sfE_n (\sfP^n_\text{in}) < \infty$. By the uniform estimates of Lemma \ref{lm:fke_uniform} we obtain 
\[
      \limsup_{n\to \infty} \sup_{t\in [0,T]} \sfE_n(\sfP_t^n) < \infty, \qquad
      \limsup_{n\to \infty } \frakd(\sfP_t^n,\sfP_s^n)\leq K |t-s|,
\]
for some constant $K>0$, independent of $s,t\in [0,T]$. Recall, $\frakd$ is the metric defined in \eqref{eq:fke_dist}, which is narrowly lower semicontinuous and whose convergence implies narrow convergence on narrowly compact subsets. Since the family $\{\sfE_n\}_{n\geq 1}$ is narrowly equicoercive (cf.\ Proposition \ref{prop:energy_lim}), an Arzel\`a–Ascoli argument gives a (not relabelled) subsequence $\sfP^n$ and a curve $\sfP$ such that $\sfP_t^n\rightharpoonup \sfP_t$ narrowly for all $t\in [0,T]$. 

The narrow convergence $\sfP^n_t \rightharpoonup \sfP_t$ and Assumption~\ref{ass:db} provides the following limits 
\begin{equation}\label{eq:conv_vp}
\begin{gathered}
   \vartheta^{n,\sfb}_{\sfP^n_t} \rightharpoonup^* \vartheta^{\sfb}_{\sfP_t}, \qquad \vartheta^{n,\sfd}_{\sfP^n_t} \rightharpoonup^* \vartheta^{\sfd}_{\sfP_t}, \qquad \vartheta^{n,\sfh}_{\sfP^n_t} \rightharpoonup^* \vartheta^{\sfh}_{\sfP_t} \\
   \sfT^{n,-}_{\#}\vartheta^{n,\sfd}_{\sfP^n_t} \rightharpoonup^* \vartheta^{\sfd}_{\sfP_t}, \qquad \sfT^{n,\sfh}_{\#} \vartheta^{n,\sfh}_{\sfP^n_t} \rightharpoonup^* \sfT^{\sfh}_{\#} \vartheta^{\sfh}_{\sfP_t}.
\end{gathered} 
\end{equation}
Namely, for any $f \in C_c(\varGamma\times X)$, dropping the $t$-dependence,
    \begin{align*}
        &\lim_{n\to\infty} \iint_{\varGamma\times X} f(\nu, x)\, \vartheta^{n,\sfb}_{\sfP^n} (\dd \nu\dd x)
        \\
        &\hspace{6em}= \lim_{n\to\infty}  \iint_{\varGamma\times X} f(\nu, x) \sfb(\nu, \dd x) \sfP^n(\dd \nu) + \lim_{n\to\infty} \iint_{\varGamma\times X} f(\nu, x) (\sfb_n - \sfb) (\nu, \dd x) \sfP^n(\dd \nu)\\
        &\hspace{6em}= \iint_{\varGamma\times X} f(\nu, x) \sfb(\nu, \dd x) \sfP(\dd \nu).
    \end{align*}
    Thus, $\vartheta^{n,\sfb}_{\sfP^n} \rightharpoonup^* \vartheta^{\sfb}_{\sfP}$ in $\calM(\varGamma\times X)$. By a similar argument, we deduce $\vartheta^{n,\sfd}_{\sfP^n} \rightharpoonup^* \vartheta^{\sfd}_{\sfP}$ in $\calM(\varGamma\times X)$ and $\vartheta^{n,\sfh}_{\sfP^n} \rightharpoonup^* \vartheta^{\sfh}_{\sfP}$ in $\calM(\varGamma\times X^2)$. Moreover, the uniform continuity of $f\in C_c(\varGamma\times X)$ provides
    \begin{align*}
        \lim_{n\to\infty} \iint_{\varGamma\times X} f(\nu, x) &\,\sfT^{n,-}_{\#}\vartheta^{n,\sfd}_{\sfP^n}(\dd \nu\dd x)
        = \lim_{n\to\infty} \iint_{\varGamma\times X} f(\nu - n^{-1}\delta_x, x) \, \vartheta^{n,\sfd}_{\sfP^n}(\dd \nu \dd x)\\
        &= \iint_{\varGamma\times X} f(\nu, x) \,\vartheta^{\sfd}_{\sfP}(\dd \nu \dd x) + \lim_{n\to\infty} \iint_{\varGamma\times X} \big( f(\nu - n^{-1}\delta_x, x) - f(\nu, x) \big) \,\vartheta^{n,\sfd}_{\sfP^n}(\dd \nu \dd x) \\
        &= \iint_{\varGamma\times X} f(\nu, x) \,\vartheta^{\sfd}_{\sfP}(\dd \nu\dd x).
    \end{align*}
    Consequently, we have convergence $\sfT^{n,-}_{\#}\vartheta^{n,\sfd}_{\sfP^n} \rightharpoonup^* \vartheta^{\sfd}_{\sfP}$ weakly-* in $\calM(\varGamma\times X)$. Similarly, we find
    \begin{align*}
        \lim_{n\to\infty} \iiint_{\varGamma\times X^2} f(\nu, x, y) & \big( \sfT^{n,\sfh}_{\#} \vartheta^{n,\sfh}_{\sfP^n} \big) (\dd\nu \dd x \dd y)
        = \lim_{n\to\infty} \iiint_{\varGamma\times X^2} f\big( \nu + n^{-1}(\delta_y - \delta_x), y, x \big) \vartheta^{n,\sfh}_{\sfP^n} (\dd\nu \dd x \dd y) \\
        &= \iiint_{\varGamma\times X^2} f(\nu, y, x) \,\vartheta^{\sfh}_{\sfP} (\dd\nu \dd x \dd y)
        = \iiint_{\varGamma\times X^2} f(\nu, x, y) \,\big( \sfT_{\#} \vartheta^{\sfh}_{\sfP} \big) (\dd\nu \dd x \dd y).
    \end{align*}
    By the representation of the fluxes $\sfJ^{n,\sfb\sfd}$ and $\sfJ^{n,h}$ as in Theorem \ref{th:fke_main} we obtain for all $t\in [0,T]$ 
    \[
    \sfJ^{n,\sfb\sfd}_{\sfP^n_t}\rightharpoonup^*\sfJ^{\sfb\sfd}_{\sfP}\coloneq \vartheta^{\sfb}_{\sfP_t}-\vartheta^{\sfd}_{\sfP_t}    , \qquad \sfJ^{n,\sfh}_{\sfP^n_t}\rightharpoonup^*\sfJ^{\sfh}\coloneq  \frac{1}{2}\left(\vartheta^{\sfh}_{\sfP_t}-\sfT_{\#}\vartheta^{\sfh}_{\sfP_t}\right).
    \]
Recall from Remark \ref{rem:li_mb} that for any $F\in \mathrm{Cyl}(\varGamma)$ the mass $\|\nu\|$ is bounded on the support of $F$, and in particular $F\in C_c(\varGamma)$. It is then clear that $\nabla_{\varGamma}F(\nu,x)\in C_c(\varGamma\times X)$, $\dnabla_{\varGamma}F(\nu,x,y)\in C_c(\varGamma\times X^2)$, and via Taylor expansion as mentioned in Section \ref{sec:liouville}, we can find a constant $c_{F}$ such that for all $\nu\in \varGamma$ and $x,y\in X$:
\[\left| \nabla_{\varGamma}F(\nu,x)-\dnabla_{n,\sfb\sfd} F(\nu,x)\right| + \left| \dnabla_{\varGamma}F(\nu,x,y)-\dnabla_{n,\sfh}F(\nu,x,y)\right|\leq c_{F} n^{-1}.\]
Therefore one can pass to the limit in the continuity equation \eqref{eq:CElim} to conclude that $(\sfP,\sfJ^{\sfb\sfd},\sfJ^{\sfh})\in \mathsf{CE}_\infty$.

\medskip

We now turn to the $\liminf$ inequality in (2). Assume that the convergences in (1) hold. We then establish separately $\liminf$ inequalities for all the components of the energy-dissipation functional, i.e.,
    \begin{gather*}
        \liminf_{n\to\infty} \int_0^T \sfR_n(\sfP^n_t, \sfJ^n_t) \dd t \geq \int_0^T \sfR_\infty(\sfP_t, \sfJ_t) \dd t,\qquad
        \liminf_{n\to\infty} \int_0^T \sfD_n(\sfP^n_t) \dd t \geq \int_0^T \sfD_\infty(\sfP_t) \dd t, \\
        \liminf_{n\to\infty} \sfE_n(\sfP^n_t) \geq 
        \sfE_\infty (\sfP_t) \qquad \text{for all } t\in [0, T].
    \end{gather*} 
Note that the $\liminf$ inequality for $\sfE_n$ follows from Proposition \ref{prop:energy_lim}. As for the dissipation potential, recall that $\sfR_n$ is the sum of two components (we omit the subscript $t$ for now):
    \begin{align*}
        \sfR_{n,\sfb\sfd}(\sfP^n,\sfJ^{n,\sfb\sfd}) = 2\iint_{\varGamma\times X} \Upsilon(\sfJ^{n,\sfb\sfd}/2,\vartheta^{n,\sfb}_{\sfP^n},\sfT^{n,-}_{\#}\vartheta^{n,\sfd}_{\sfP^n}), \qquad  \sfR_{n,\sfh} (\sfP^n, \sfJ^{n,\sfh}) = \iiint_{\varGamma\times X^2} \Upsilon(\sfJ^{n,\sfh},\vartheta^{n,\sfh}_{\sfP^n},\sfT^{n,\sfh}_{\#}\vartheta^{n,\sfh}_{\sfP^n}).
    \end{align*}
Due to the convergences in \eqref{eq:conv_vp} and convergence of the fluxes, we can use the convexity and lower semi-continuity of $\Upsilon$ and a standard relaxation of integral functionals argument (see e.g.\ \cite[Theorem 3.4.3]{Buttazzo1989}) to conclude that
    \begin{align*}
        \liminf_{n\to\infty} \sfR_{n,\sfb\sfd}(\sfP^n,\sfJ^{n,\sfb\sfd}) \geq 2\iint_{\varGamma\times X} \Upsilon(\sfJ^{\sfb\sfd}/2,\vartheta^{\sfb}_{\sfP}, \vartheta^{\sfd}_{\sfP}), \qquad
        \liminf_{n\to\infty} \sfR_{n,\sfh} (\sfP^n, \sfJ^{\sfh}) \geq \iiint_{\varGamma\times X^2} \Upsilon(\sfJ^{\sfh},\vartheta^{\sfh}_{\sfP},\sfT_{\#}\vartheta^{\sfh}_{\sfP}).
    \end{align*}
The $\liminf$ inequality for the Fisher information follows similarly, from the convexity and lower semicontinuity of the squared Hellinger distance. 

\medskip
Finally, to establish (3), note that $\mathscr{I}_n([0,T];\sfP^n,\sfJ^n)=0$. By the lower semicontinuity results of (2) we find that $\mathscr{I}_{\infty}([0,T];\sfP,\sfJ)\leq 0$, and conclude by Theorem \ref{th:superposition} that $(\sfP,\sfJ)$ is the unique generalized gradient flow solution of the Liouville equation \eqref{eq:Liouville} with initial datum $\sfP_{in}$.
\end{proof}

With Theorem \ref{th:convergence} in hand, the main result of this paper follows after classical compactness and lower semicontinuity argument (c.f. \cite{sandier2004gamma,serfaty2011gamma}).

\begin{proof}[Proof of Theorem~\ref{th:main}]
Since all converging subsequences of $(\sfP^n,\sfJ^n)$ converge to the unique generalized flow of the Liouville equation, we obtain by compactness the convergence of the original sequence $(\sfP^n,\sfJ^n)$. Moreover, from the convergence $\sfE_n(\sfP^n_\text{in})\to \sfE_\infty(\sfP_\text{in})$ and our previously established lower semicontinuity results, we find that for every $t\in [0,T]$, 
\begin{align*}
    \limsup_{n\to \infty} \sfE_n(\sfP^n_t)&\leq\lim_{n\to \infty}\sfE_n(\sfP^n_\text{in})-\liminf_{n\to\infty} \int_0^t \left(  \sfR_n(\sfP^n_r, \sfJ^n_r) + \sfD_n(\sfP^n_r)\right) \dd r \\
    &\leq \sfE_{\infty}(\sfP_\text{in})-\int_0^t \left(  \sfR_{\infty}(\sfP_r, \sfJ_r) + \sfD_{\infty}(\sfP_r)\right) \dd r = \sfE_{\infty}(\sfP_t) \le \liminf_{n\to \infty} \sfE_n(\sfP^n_t),
\end{align*}
and hence we obtain convergence of the energies $\sfE_n$. Substituting $\sfP^n_\text{in}=\delta_{\bar \nu}$, and noting that Theorem \ref{th:superposition} ensures that $\sfP_t=\nu_t$ with $\nu$ being the unique solution to \eqref{eq:MF}, the convergence of $\sfE_n(\sfP_t^n)$ to $\sfE(\sfP_t)=\calE(\nu_t)$ follows. Using the uniform upper and lower bounds of $u_t$, one can follow the arguments in the proof of \cite[Theorem 5.4]{hoeksema2023} to transform the latter convergence into the desired entropic propagation of chaos result, i.e.,
\[
    \lim_{n\to\infty} \frac{1}{n}\Ent(\sfP_t^n|\Pi_{\nu_t}^n) = 0\qquad\text{for every $t\in[0,T]$}.\qedhere
\]
\end{proof}

\bibliographystyle{abbrv}
\bibliography{ref}

\end{document}